\documentclass[twoside,11pt,reqno]{amsart}
\usepackage{amsmath,amssymb,amscd,mathrsfs,amscd}
\usepackage{graphics,verbatim}

\newcommand{\p}[1]{\ensuremath{\overline{#1}}}
\newcommand{\losemi}{{\otimes \kern -.78em \ltimes}}
\newcommand{\rosemi}{{\otimes \kern -.78em \rtimes}}

\newcommand{\Hom}{\ensuremath{\operatorname{Hom}}}

\newcommand{\Res}{\ensuremath{\operatorname{Res}}}

\newcommand{\Ker}{\ensuremath{\operatorname{Ker} }}

\newcommand{\Ext}{\operatorname{Ext}}

\newcommand{\0}{\bar 0}
\newcommand{\1}{\bar 1}
\newcommand{\Z}{\mathbb{Z}}
\newcommand{\C}{\mathbb{C}}

\newcommand{\setof}[2]{\ensuremath{\left\{ #1 \:|\: #2 \right\}}}
\newcommand{\gl}{\ensuremath{\mathfrak{gl}}}
\newcommand{\g}{\ensuremath{\mathfrak{g}}}
\newcommand{\e}{\ensuremath{\mathfrak{e}}}

\newcommand{\res}{\ensuremath{\operatorname{res}}}

\newcommand{\fg}{\ensuremath{\mathfrak{g}}}
\newcommand{\fb}{\ensuremath{\mathfrak{b}}}

\newcommand{\fh}{\ensuremath{\mathfrak{h}}}

\newcommand{\fe}{\ensuremath{\e}}
\newcommand{\fc}{\ensuremath{\mathfrak{c}}}
\newcommand{\fa}{\ensuremath{\mathfrak{a}}}

\newcommand{\fp}{\ensuremath{\mathfrak{p}}}
\newcommand{\ft}{\ensuremath{\mathfrak{t}}}

\newcommand{\atyp}{\ensuremath{\operatorname{atyp}}}

\newcommand{\HH}{\operatorname{H}}

\newtheorem{Df}{Definition}[subsection]
\newtheorem{theorem}[Df]{Theorem}

\newtheorem{corollary}[Df]{Corollary}
\newtheorem{prop}[Df]{Proposition}

\numberwithin{equation}{subsection}

\begin{document}
\title{Cohomology and Support Varieties for Lie Superalgebras II}

\author{Brian D. Boe }
\address{Department of Mathematics \\
            University of Georgia \\
            Athens, GA 30602}
\email{brian@math.uga.edu}
\author{Jonathan R. Kujawa}
\address{Department of Mathematics \\
            University of Oklahoma \\
            Norman, OK 73019}
\thanks{Research of the second author was partially supported by NSF grant
DMS-0402916}\
\email{kujawa@math.ou.edu}
\author{Daniel K. Nakano}
\address{Department of Mathematics \\
            University of Georgia \\
            Athens, GA 30602}
\thanks{Research of the third author was partially supported by NSF
grant  DMS-0400548}\
\email{nakano@math.uga.edu}
\date{\today}

\begin{abstract} In \cite{BKN} the authors initiated a study of the representation theory of classical 
Lie superalgebras via a cohomological approach. Detecting subalgebras were constructed and a 
theory of support varieties was developed. The dimension of a detecting subalgebra coincides  
with the defect of the Lie superalgebra and the dimension of the support variety  
for a  simple supermodule was conjectured to equal the atypicality of the supermodule. In this paper 
the authors compute the support varieties for Kac supermodules for Type I Lie superalgebras and 
the simple supermodules for $\mathfrak{gl}(m|n)$. The latter result verifies our earlier conjecture 
for $\mathfrak{gl}(m|n)$. In our investigation we also delineate several of the major differences between Type I 
versus Type II classical Lie superalgebras. Finally, the connection between atypicality, defect and superdimension is 
made more precise by using the theory of support varieties and representations of Clifford superalgebras.  
\end{abstract}

\maketitle

\section{Introduction}\label{S:intro}  

\subsection{} Let ${\mathfrak g}={\mathfrak g}_{\0}\oplus {\mathfrak g}_{\1}$ be a 
classical Lie superalgebra over the complex numbers. For any classical Lie superalgebra there exists by definition a connected reductive algebraic group $G_{\0}$  such that $\operatorname{Lie}\left(G_{\0} \right)
={\mathfrak g}_{\0}$. The simple classical Lie superalgebras were classified by Kac \cite{Kac1}. 
In \cite{BKN} the authors used relative cohomology for the pair $({\mathfrak g},{\mathfrak g}_{\0})$
to investigate the  combinatorics and representation theory of the blocks in the category of  finite dimensional 
representations of the Lie superalgebra ${\mathfrak g}$. In this situation the cohomology 
ring $R=\HH^{\bullet}({\mathfrak g},{\mathfrak g}_{\0}; {\mathbb  C})$ is 
finitely generated because $G_{\0}$ is reductive.  
By using invariant theoretic results due to Luna and Richardson \cite{luna} and 
Dadok and Kac \cite{dadokkac}, the authors were able to construct natural ``detecting'' subalgebras 
${\mathfrak e}={\mathfrak e}_{\0}\oplus {\mathfrak e}_{\1}$ of ${\mathfrak g}$ 
such that the restriction map in cohomology induces an isomorphism 
\[
R=\HH^{\bullet}(\fg,\fg_{\0};\C )
\cong \HH^{\bullet}(\fe,\fe_{\0};\C)^{W},
\] 
where $W$ is a finite pseudoreflection group. One striking outcome of our construction was that 
the dimension of the odd part of the detecting subalgebras and the Krull dimension of $R$ both coincide with the 
combinatorially defined defect of ${\mathfrak g}$ introduced earlier by 
Kac and Wakimoto \cite{kacwakimoto}. 

\subsection{}\label{SS:intro2} Given a finite dimensional ${\mathfrak g}$-supermodule, $M$, one can use the finite generation 
of $R$ to define the cohomological support varieties ${\mathcal V}_{({\mathfrak e},{\mathfrak e}_{\0})}(M)$ and 
${\mathcal V}_{({\mathfrak g},{\mathfrak g}_{\0})}(M)$. In \cite[Theorem 6.2.2]{BKN} it was 
proved that ${\mathcal V}_{({\mathfrak g},{\mathfrak g}_{\0})}(M)$ can be identified 
generically with ${\mathcal V}_{({\mathfrak e},{\mathfrak e}_{\0})}(M)/W$ and conjectured that this 
should hold everywhere. The variety ${\mathcal V}_{({\mathfrak e},{\mathfrak e}_{\0})}(M)$ 
can be identified via a ``rank variety'' description as a certain subvariety of 
${\mathfrak e}_{\1}$ \cite[Theorem 6.3.2]{BKN}. The rank variety description enabled us to 
demonstrate that the representation theory for the superalgebra over ${\mathbb C}$ 
has similar features to modular representations of finite groups over 
fields of characteristic two (cf.\ \cite[Corollary 6.4.1]{BKN}). 
When $\fg$ admits an appropriate bilinear form one can define the ``atypicality'' of a block 
and of a simple supermodule.  Atypicality, due to Kac and Serganova, is a combinatorial 
invariant used to give a rough measure of the complications involved in the block structure. 
Evidence from examples led us to conjecture that if $L(\lambda)$ is a finite dimensional 
simple ${\mathfrak g}$-supermodule then the atypicality of $L(\lambda)$ 
(denoted by $\text{atyp}(L(\lambda))$) equals $\dim {\mathcal V}_{({\mathfrak e},
{\mathfrak e}_{\0})}(L(\lambda))$ (cf.\ \cite[Conjecture 7.2.1]{BKN}). 

\subsection{} This paper is aimed at providing applications and concrete computations for 
the theory developed in \cite{BKN}. In Section~\ref{S:Kacmodules} we distinguish 
between Type I and Type II classical Lie superalgebras. Type I Lie superalgebras are the 
ones (such as $\mathfrak{gl}(m|n)$) which admit a compatible $\Z$-grading concentrated in degrees $-1,$ $0,$ and $1.$  Otherwise $\fg$ is said to be of Type II.  We prove for Type I Lie superalgebras that the category $\mathcal{F}$ of finite dimensional supermodules is a highest weight category as defined by Cline, Parshall and Scott \cite{CPS}. 
We also show that if $K(\lambda)$ is a Kac supermodule (i.e.,\ universal highest weight 
supermodule) for a Type I Lie superalgebra $\fg,$ then ${\mathcal V}_{(\fg,\fg_{\0})}(K(\lambda))=
{\mathcal V}_{(\fe,\fe_{\0})}(K(\lambda))=\{0\}$. On the other hand, by using work 
of Germoni \cite{germoni} on the Type II classical Lie superalgebras $\mathfrak{osp}(3|2),$ $D(2,1;\alpha),$ and $G(3),$ one 
sees that when $\fg$ is of Type II the category ${\mathcal F}$ need not be a highest weight category 
and there exist Kac supermodules with nontrivial support varieties.  These results make clear the significant representation theoretic differences between Type I and Type II Lie superalgebras.

In Section~\ref{S:typeA} we apply results of \cite{BKN}, Duflo and Serganova \cite{dufloserganova}, and Serganova \cite{serganova3} to compute the support varieties ${\mathcal V}_{({\mathfrak g},{\mathfrak g}_{\0})}(L(\lambda))$ 
and ${\mathcal V}_{({\mathfrak e},{\mathfrak e}_{\0})}(L(\lambda))$ of the finite dimensional simple supermodules $L(\lambda)$ when 
${\mathfrak g}=\mathfrak{gl}(m|n)$. For simple $\mathfrak{gl}(m|n)$-supermodules 
our computations verify the conjectures mentioned in Section~\ref{SS:intro2}. Namely, we prove 
$${\mathcal V}_{({\mathfrak g},{\mathfrak g}_{\0})}(L(\lambda))\cong 
{\mathcal V}_{({\mathfrak e},{\mathfrak e}_{\0})}(L(\lambda))/W,$$
and that the atypicality of $L(\lambda)$ equals $\dim {\mathcal V}_{({\mathfrak e},{\mathfrak e}_{\0})}
(L(\lambda))$. A remarkable outcome of our results is that one can now extend the definition  
of atypicality to {\em all} $\mathfrak{gl}(m|n)$-supermodules in a functorial way by setting 
$\text{atyp}(M)=\dim {\mathcal V}_{({\mathfrak e},{\mathfrak e}_{\0})}(M)$.  

In the final section we demonstrate that the codimension of the support variety 
${\mathcal V}_{({\mathfrak e},{\mathfrak e}_{\0})}(M)$ is directly  
related to the $2$-divisibility of the dimension of $M$. In many ways 
our results can be thought of as block theoretic analogues 
of the celebrated Kac-Weisfeiler Conjecture for non-restricted representations
of classical Lie algebras over fields of characteristic $p>0$. The 
Kac-Weisfeiler Conjecture \cite{kacweisfeiler} connects the $p$-divisibility of modules with the size of 
an associated geometric object 
(i.e., the corresponding nilpotent orbit).  Premet proved the conjecture in 1995 \cite{premet}.  We also 
relate $2$-divisibility 
to the defect and atypicality for the simple $\mathfrak{gl}(m|n)$-supermodules. 

The authors are grateful to Weiqiang Wang for bringing the paper \cite{germoni} to our attention.  
The second author would like to thank Calvin Burgoyne and Mitchell Rothstein for helpful 
conversations about the representation theory of Clifford algebras.
\newpage

\section{Notation}\label{S:prelims}  

\subsection{}\label{SS:prelims}  Throughout 
we work with the complex numbers $\C$ as the ground field.  Recall that a superspace is a $\Z_{2}$-graded 
vector space and, given a superspace $V$ and a homogeneous vector $v \in V,$ we write $\p{v} \in \Z_{2}$ 
for the \emph{parity} (or \emph{degree}) of $v.$  Elements of $V_{\0}$ (resp.\ $V_{\1}$) are called \emph{even} 
(resp.\ \emph{odd}).  The \emph{superdimension} of a superspace $V$ is the integer $\dim V_{\0}-\dim V_{\1}.$  Note that if $M$ and $M'$ are two superspaces, then the space $\Hom_{\C}(M,M')$ is 
naturally $\Z_{2}$-graded by declaring $f \in \Hom_{\C}(M,M')_{r}$ ($r\in \Z_{2}$) if $f(M_{s}) \subseteq M'_{s+r}$ for 
all $s \in \Z_{2}$.   

A superalgebra is a $\Z_2$-graded, unital, associative algebra $A=A_{\0}\oplus A_{\1}$ which satisfies 
$A_{r}A_{s}\subseteq A_{r+s}$ for all $r,s\in \Z_2.$  A \emph{Lie superalgebra} is a superspace 
$\g=\g_{\0}\oplus \g_{\1}$ with a bracket operation $[\;,\;]:\g \otimes \g \to \g$ which preserves the 
$\Z_{2}$-grading and satisfies graded versions of the usual Lie bracket axioms.  In particular, we note 
that $\g_{\0}$ is a Lie algebra under the bracket obtained by restricting the bracket of $\g.$  If $\g$ is a 
Lie superalgebra, then one has a universal enveloping superalgebra $U(\g)$ which is $\Z_2$-graded and 
satisfies a PBW-type theorem.  See, for example, \cite{Kac1} for details and further background on Lie superalgebras.

We call a finite dimensional Lie superalgebra \emph{classical} if there is a connected reductive algebraic group $G_{\0}$ such that $\operatorname{Lie}(G_{\0})=\g_{\0},$ and an action of $G_{\0}$ on $\g_{\1}$ which differentiates to the adjoint action of $\g_{\0}$ on $\g_{\1}.$  In particular, if $\g$ is classical, then $\g_{\0}$ is a reductive Lie algebra and $\g_{\1}$ is semisimple as a $\g_{\0}$-module.  Note that we do not assume that $\g$ is simple.  A \emph{basic classical} Lie superalgebra is a classical Lie superalgebra with a nondegenerate invariant supersymmetric even bilinear form.  The simple classical Lie superalgebras were classified by Kac \cite{Kac1}.

Given a Lie superalgebra $\g$, let us describe the category of $\g$-supermodules.  The objects are 
all left $U(\g)$-modules which are $\Z_{2}$-graded; that is, superspaces $M$ satisfying $U(\g)_{r}M_{s} 
\subseteq M_{r+s}$ for all $r, s \in \Z_{2}.$  If $M$ is a $\g$-supermodule, then by definition $N \subseteq M$ is a 
subsupermodule if it is a supermodule which inherits its grading from $M$ in the sense that $M_{r} 
\cap N = N_{r}$ for $r \in \Z_2$. We say a supermodule is \emph{finitely semisimple} if it decomposes into a direct sum of finite dimensional simple supermodules.  Given ${\mathfrak g}$-supermodules $M$ and $N$ one can use the 
antipode and coproduct of $U({\mathfrak g})$ to define a ${\mathfrak g}$-supermodule 
structure on the contragradient dual $M^{*}$ and the tensor product $M\otimes N$. 

A morphism of $U(\g)$-supermodules is an element $f\in\Hom_{\C}(M,M^{\prime})$ 
satisfying $f(xm)=(-1)^{\p{f}\;\p{x}}xf(m)$ for all $m \in M$ and 
all $x \in U(\g).$  Note that this 
definition makes sense as stated only for homogeneous elements; it should be interpreted via linearity 
in the general case.  We emphasize that we allow \emph{all} morphisms and not just graded (i.e.\ \emph{even}) 
morphisms.  However, note that  $\Hom_{U(\g)}(M,M^{\prime})$ inherits a $\Z_{2}$-grading as a subspace of  $\Hom_{{\mathbb C}}(M,M^{\prime}).$ 

The category of $\g$-supermodules is not an abelian category.  However, the \emph{underlying even category}, 
consisting of the same objects but only the even morphisms, is an abelian category.  This, along with the parity 
change functor, $\Pi,$ which interchanges the $\Z_{2}$-grading of a supermodule, allows one to make use of 
the tools of homological algebra (cf.\ \ref{SS:HWC2}).  

 Given a category, $\mathcal{C},$ of $\fg$-supermodules and objects $M, N$ in $\mathcal{C},$ we write 
\[
\Ext_{\mathcal{C}}^{d}(M,N)
\] for the degree $d$ extensions between $N$ and $M$ in the category $\mathcal{C}.$

As a special case of the above discussion, we always view a Lie algebra (e.g.\ the even part of a Lie superalgebra) 
as a Lie superalgebra concentrated in degree $\0$.   

\subsection{Relative Cohomology}\label{SS:relcohom} Let us recall the definition of relative cohomology for Lie superalgebras.  Let $M$ be a $\fg$-supermodule.  Let $\fg$ be a Lie superalgebra, $\ft \subseteq \fg$ a Lie subsuperalgebra, and $M$ a $\fg$-supermodule.  Define 
\[
C^{p}(\g, \mathfrak{t}; M)=\Hom_{\mathfrak{t}}(\wedge^{p}_{s}(\g/\mathfrak{t}), M).
\]  Note that $\wedge^{p}_{s}(V)$ denotes the super (i.e.\  graded) wedge product of the superspace $V$.  In particular, if $V=V_{\1}$ then $\Lambda^{p}_{s}(V) = S^{p}(V),$ the ordinary symmetric product.  

There is a differential $d^{p}: C^{p}(\g, \mathfrak{t}; M) \to  C^{p+1}(\g, \mathfrak{t}; M)$ defined as
for relative cohomology for Lie algebras (cf.\ \cite[Section 2.2]{BKN}) such that 
\[
\HH^{p}(\g, \mathfrak{t}; M)=\Ker d^{p}/\operatorname{Im} d^{p-1}.
\]  Let $\mathcal{C}(\fg, \ft)$ denote the full subcategory of all $\fg$-supermodules which are finitely semisimple as $\ft$-supermodules.  A key connection is that if $M$ and $N$ are objects of $\mathcal{C}(\fg ,\ft),$ then one has 
\begin{equation}\label{E:relcohomiso}
\Ext_{\mathcal{C}(\fg ,\ft)}^{\bullet}(M,N) \cong \HH^{\bullet}(\g, \mathfrak{t}; M^{*}\otimes N).
\end{equation}

\subsection{Support Varieties}\label{SS:morecohom} Let $\mathcal{F}=\mathcal{F}(\fg, \fg_{\0})$ denote the full subcategory of all finite dimensional $\fg$-supermodules which are completely reducible as $\fg_{\0}$-supermodules.  This is the category of interest in \cite{BKN}.  One then has the graded cohomology ring 
\[
R:=\Ext^{\bullet}_{\mathcal{F}}(\C,\C).
\] 
 It was proven in \cite[Theorem 2.5.1]{BKN} that if $M$ and $N$ are objects in $\mathcal{F},$ then 
\begin{equation}\label{E:relcohomiso2}
\HH^{\bullet}(\fg ,\fg_{\0}; M^{*} \otimes N) \cong \Ext^{\bullet}_{\mathcal{F}}(M,N).
\end{equation}
As a consequence, when $\fg$ is classical $R$ is a finitely generated commutative ring.  Furthermore, if $M$ and $N$ are objects of $\mathcal{F},$ then $\Ext^{\bullet}_{\mathcal{F}}(M,N)$ is a finitely generated graded $R$-module \cite[Theorem 2.5.3]{BKN}.

Set 
\[
I_{(\mathfrak{g},\mathfrak{g}_{\0})}(M,N)=\operatorname{Ann}_{R}(\Ext_{\mathcal{F}}^{\bullet}(M,N)),
\] the annihilator ideal of this module.  We define the \emph{support variety of the pair $(M,N)$} to be 
\begin{equation}\label{E:relsuppvardef} \mathcal{V}_{(\mathfrak{\fg },\mathfrak{\fg }_{\0})}(M,N)=
\operatorname{MaxSpec}(R/I_{(\mathfrak{g},\mathfrak{g}_{\0})}(M,N)).
\end{equation}  In particular, when $M=N$ we define the \emph{support variety of $M$} to be 
\begin{equation}\label{E:suppvardef} \mathcal{V}_{(\mathfrak{\fg },\mathfrak{\fg }_{\0})}(M)=
\operatorname{MaxSpec}(R/I_{(\mathfrak{g},\mathfrak{g}_{\0})}(M,M)).
\end{equation}

In the case when $\fg$ is a simple classical Lie superalgebra, the ring $R$ turns out to always be a polynomial ring \cite[Section 8.8]{BKN} in, say, $r$ variables.  In this case one can view $\mathcal{V}_{(\mathfrak{\fg },\mathfrak{\fg }_{\0})}(M)$ as a closed, conical affine variety; namely, the subvariety of $\operatorname{MaxSpec}(R) \cong \mathbb{A}^{r}$ defined by the ideal $I_{(\mathfrak{g},\mathfrak{g}_{\0})}(M,M).$

\subsection{Detecting Subalgebras}\label{SS:detecting}  If $\fg$ is a 
classical Lie superalgebra which satisfies certain invariant theoretic restrictions (i.e.\ $\fg$ is \emph{stable} and \emph{polar}), then one can construct a classical subalgebra $\fe \subseteq \fg$ such 
that the natural inclusion map $\fe \hookrightarrow \fg$ induces the following isomorphism in cohomology,
\[
R \cong \HH^{\bullet}(\fe , \fe_{\0};\C)^{W},
\] where $W$ is a finite pseudoreflection group \cite[Section 4.1]{BKN}.  

As $\fe$ is classical, one can define support varieties just as in \eqref{E:suppvardef} for objects of $\mathcal{F}(\fe , \fe_{\0}).$   The variety $\mathcal{V}_{(\fe ,\fe_{\0})}(M)$ shares a number of features of the classical theory of support varieties and admits the following rank variety description \cite[Theorem 6.3.2]{BKN}.  As a matter of notation, given a homogeneous 
element $x \in \fe,$ let $\langle x \rangle$ denote the Lie subsuperalgebra generated by $x.$   
Define the \emph{rank variety} of $M$ to be 
\[
{\mathcal V}_{(\fe ,\fe_{\0})}^{\text{rank}}\left(M \right)=\left\{ x \in\fe_{\1}\:\vert\:  M
\text{ is not projective as an $\langle x \rangle$-supermodule} \right\} \cup \{0 \}.
\] Then ${\mathcal V}_{(\fe ,\fe_{\0})}(M) \cong {\mathcal V}_{(\fe ,\fe_{\0})}^{\text{rank}}(M).$
We also remark that if $M$ is an object of $\mathcal{F}(\fg , \fg_{\0}),$ then one can view $M$ as an object of $\mathcal{F}(\fe , \fe_{\0})$ by restriction.  In particular, one then has the following induced map of varieties,
\[
\operatorname{res}^{*} : \mathcal{V}_{(\e,\e_{\0})}(M) \to  \mathcal{V}_{(\g,\g_{\0})}(M).
\]  Furthermore, one has (cf.\ \cite[$(6.1.3)$]{BKN}),
\begin{align}\label{E:resmapsbetweenvarieties}
\mathcal{V}_{(\e,\e_{\0})}(M)/W &\cong \res^{*}\left(\mathcal{V}_{(\e,\e_{\0})}(M)\right)\subseteq \mathcal{V}_{(\g,\g_{\0})}(M).
\end{align}

\section{Type I and Type II Lie Superalgebras}\label{S:Kacmodules}
A Lie superalgebra is said to be of Type I if it admits a consistent $\Z$-grading concentrated in degrees $-1,$ $0,$ and $1.$  Otherwise it is of Type II (cf.\ Section~\ref{SS:Kacasses}).  The distinction between these two types was first made by Kac \cite{Kac1}.  In this section we will see that this distinction is also significant from the point of view of representation theory and support varieties. 


\subsection{}\label{SS:Kacasses} Let ${\mathfrak g}={\mathfrak g}_{\0}\oplus 
{\mathfrak g}_{\1}$ be a classical Lie superalgebra  with a $\Z$-grading 
\[
\fg=\fg_{-1}\oplus \fg_{0}\oplus \fg_{1}
\] which satisfies $[\fg_{i},\fg_{j}] \subseteq \fg_{i+j}$ for all $i,j \in \Z,$ $\fg_{\0}= \fg_{0},$ and $\fg_{\1}= \fg_{-1}\oplus \fg_{1}.$  That is, in the terminology of \cite{Kac1}, the $\Z$-grading is \emph{consistent} since $\fg_{\0}=\sum_{k}\fg_{2k}$ and $\fg_{\1}=\sum_{k}\fg_{2k+1}$.  As $\fg$ is concentrated in degrees $-1$, $0$, and $1,$ by definition $\fg$ is of Type I.    Let 
\begin{align*}
 \fp^{+}= \fg_{0} \oplus \fg_{1} && \text{and} &&  \fp^{-}= \fg_{0} \oplus \fg_{-1} .
\end{align*}
We observe that, because of the $\Z$-grading, $\fg_{1}$ is an ideal of $\fp^{+}.$
Fix a Cartan subalgebra $\fh  \subseteq \fg_{\0}$ and Borel subalgebras $\fb_{\0} \subseteq \fg_{\0}$ and $\fb \subseteq \fg$ so that $\fh  \subseteq \fb_{\0}$ and 
\begin{equation}\label{E:borelequality}
\fb = \fb_{\0} \oplus \fg_{1}.
\end{equation}  Let $\leq$ denote the usual dominance order on $\fh^{*}$ given by the choice of $\fb.$  Note that the Lie superalgebra $\gl (m|n)$ and the simple Lie superalgebras of types $A(m,n)$ and $C(n)$ are all of Type I (cf.\ \cite[Section 2]{Kac1}).

 Let $X^{+}_{\0} \subseteq \fh^{*}$ denote the parameterizing set of highest weights for the simple finite dimensional $\fg_{\0}$-supermodules with respect to the pair $(\fh, \mathfrak{b}_{\0})$.  Given $\lambda \in X^{+}_{\0},$ let $L_{\0}(\lambda)$ denote the simple $\fg_{\0}$-supermodule of highest weight $\lambda$ (concentrated in degree $\0$).  We view $L_{\0}(\lambda)$ as a simple $\fp^+$-supermodule via inflation through the canonical quotient $\fp^{+} \twoheadrightarrow \fg_{\0}.$  Let
\[
K(\lambda)=U(\fg)\otimes_{U(\fp^{+})} L_{\0}(\lambda),
\] be the \emph{Kac supermodule} of highest weight $\lambda.$  A standard argument using \eqref{E:borelequality} and the fact that $\fg_{1}$ is an ideal of $\fp^{+}$ shows that $K(\lambda)$ is the universal highest weight supermodule 
in $\mathcal{F}$ of highest weight $\lambda$.  By highest weight arguments, $K(\lambda)$ has simple head $L(\lambda)$ and the set 
\[
\setof{L(\lambda)}{\lambda \in X_{\0}^{+}}
\] is a complete irredundant collection of simple objects in $\mathcal{F}.$  Note that since $L(\lambda)$ is generated by a one dimensional highest weight space, it follows that $\Hom_{\fg}(L(\lambda),L(\lambda))$ is one dimensional and consists of maps given by scalar multiplication.

\subsection{}\label{SS:danstheorem}  The following theorem shows that the relative cohomology of a Kac supermodule can be nonzero in only 
finitely many degrees.  We remind the reader of the relevant notation and results which were established in Section~\ref{SS:relcohom}.

\begin{theorem}\label{T:danstheorem}Let $\lambda \in X^{+}_{\0}$ and $M\in \mathcal{F}$. Then 
$\Ext^{n}_{\mathcal{F}}(K(\lambda),M)=0$ 
for $n\gg 0$. 
\end{theorem} 

\begin{proof} We first require some preliminary results.  By \eqref{E:relcohomiso},~\eqref{E:relcohomiso2}, and Frobenius reciprocity (cf.\ \cite[Lemma 3.1.14]{kumar}) we have 
\begin{equation}\label{E:daniso1}
\Ext^{j}_{\mathcal{F}}(K(\lambda), M) \cong \Ext^{j}_{\mathcal{C}(\fg,\fg_{\0})}(K(\lambda),M)\cong  
\Ext^{j}_{\mathcal{C}(\fp^{+},\fg_{\0})}\left( L_{\0}(\lambda),M \right)
\end{equation}
for 
$j\geq 0$. Since $\fg_{1}$ is an ideal of $\fp^{+},$ one can apply the Lyndon-Hochschild-Serre spectral sequence 
for the pair $(\fg_{1}, \{0 \})$ in $(\fp^{+}, \fg_{\0})$ (cf.\ \cite[Theorem 6.5]{BorelWallach}): 
\[
E_{2}^{i,j}=\Ext^{i}_{\mathcal{C}(\fg_{\0},\fg_{\0})}(L_{\0}(\lambda),
\Ext^{j}_{\mathcal{C}(\fg_{1},\{0\})}(\C,M))\Rightarrow   
\Ext^{i+j}_{\mathcal{C}(\fp^{+},\fg_{\0})}(L_{\0}(\lambda),M).
\]  Since $\Ext_{\mathcal{C}(\fg_{\0},\fg_{\0})}^{i}(-,-)$ calculates extensions in the category of $\fg_{\0}$-supermodules which are completely reducible, the higher $\Ext$'s vanish and the spectral sequence collapses.  This yields the following isomorphism:

\begin{equation}\label{E:daniso2}
\Hom_{\fg_{\0}}\left(L_{\0}(\lambda), \Ext^{j}_{\mathcal{C}(\fg_{1},\{0\})}(\C,M) \right) \cong   \Ext^{j}_{\mathcal{C}(\fp^{+},\fg_{\0})}\left(L_{\0}(\lambda),M\right),
\end{equation} for all $j\geq 0.$
Combining \eqref{E:daniso1} and \eqref{E:daniso2}, we obtain
\begin{equation}\label{E:daniso3}
\Ext^{j}_{\mathcal{F}}(K(\lambda), M) \cong
\Hom_{\fg_{\0}}\left(L_{\0}(\lambda), \Ext^{j}_{\mathcal{C}(\fg_{1},\{0\})}({\mathbb C},M) \right), 
\end{equation} for all $j\geq 0.$

By the definition of relative cohomology for the pair $(\fg_{1}, \{0 \}),$ $\Ext^{j}_{\mathcal{C}(\fg_{1},\{0\})}({\mathbb C},M)$ is a subquotient of $S^{j}(\fg_{1}^{*}) \otimes M.$  We observe that since $M$ is finite dimensional it has only finitely many nonzero weight spaces.  Furthermore, by our assumptions in Section~\ref{SS:Kacasses}, $\fg_{1}$ consists of positive root spaces so $\fg_{1}^{*}$ has a weight space decomposition consisting of negative roots.  Therefore for $j$ sufficently large the $\fg_{\0}$-supermodule $S^{j}(\fg_{1}^{*}) \otimes M$ cannot have a nontrivial $\lambda$ weight space.  That is, for $j$ sufficently large $L_{\0}(\lambda)$ does not appear as a composition factor of $S^{j}(\fg_{1}^{*})\otimes M$ as a $\fg_{\0}$-supermodule.  Hence it also does not appear in $\Ext^{j}_{\mathcal{C}(\fg_{1},\{0\})}({\mathbb C},M)$ and, therefore, for $j \gg 0$ the $\Hom$ space in \eqref{E:daniso3} is zero.  This implies the desired result.
\end{proof} 

\subsection{Support Varieties for Kac Supermodules} The previous theorem allows one to compute the relative support varieties and 
support varieties of Kac supermodules. We continue with the assumptions of the previous section.

\begin{corollary}\label{C:kacvariety} Let $K(\lambda)$ be a Kac supermodule of highest weight $\lambda \in X^{+}_{\0}$ and let $M\in \mathcal{F}$.  Then, 
\begin{itemize} 
\item[(a)] $\mathcal{V}_{(\g,\g_{\0})}(K(\lambda),M)=\{0\}$; 
\item[(b)] $\mathcal{V}_{(\g,\g_{\0})}(K(\lambda))=\{0\}$; 
\item[(c)] $\mathcal{V}_{({\mathfrak e},{\mathfrak e}_{\0})}(K(\lambda))=\{0\}$;
\item[(d)] $K(\lambda)$ is projective upon restriction to $U({\mathfrak e})$.
\end{itemize}   
\end{corollary} 

\begin{proof}  We first prove (a).  By Theorem~\ref{T:danstheorem} we can fix $N \geq 0$ so that $\Ext_{\mathcal{F}}^{j}(K(\lambda),M)=0$ for all $j \geq N.$  Then, given an element $x \in R$ of positive degree, it follows that $x^{N}$ annihilates  $\Ext_{\mathcal{F}}^{\bullet}(K(\lambda),M)$ as it is a graded $R$-module.  That is, all postive degree elements of $R$ are contained in the radical of the ideal
\[
I_{(\mathfrak{g},\mathfrak{g}_{\0})}\left(K(\lambda),M \right)=\operatorname{Ann}_{R}\left( \Ext_{\mathcal{F}}^{\bullet}(K(\lambda),M)\right).
\] However, an ideal and its radical ideal define the same variety.  Therefore, by \eqref{E:relsuppvardef}, $\mathcal{V}_{(\g,\g_{\0})}(K(\lambda),M)=\{0\}.$  Part (b) follows immediately from (a) by setting $M=K(\lambda).$  To prove (c), one notes that the map in \eqref{E:resmapsbetweenvarieties},  $\res^{*}:\mathcal{V}_{(\fe, \fe_{\0})}(K(\lambda)) \to \mathcal{V}_{(\fg, \fg_{\0})}(K(\lambda)),$ is finite-to-one.  Since  $\mathcal{V}_{({\mathfrak e},{\mathfrak e}_{\0})}(K(\lambda))$ is a conical variety, it follows that it must be $\{0 \}$.  Finally, part (d) follows from \cite[Theorem 6.4.2(b)]{BKN}.
\end{proof}

In the case of modular representations of finite groups and restricted Lie algebras,  one has that the support variety is trivial if and only if the module is projective.  In contrast, despite Corollary~\ref{C:kacvariety}(b), $K(\lambda)$ need not be a projective $\fg$-supermodule.  For example, this is already true in the principal block of $\gl(1|1).$ 

\subsection{Filtrations by Kac Supermodules}\label{SS:HWC1} We now show that for a Type I Lie superalgebra the category $\mathcal{F}$ is a highest weight category in the sense of \cite{CPS}.  This was proven for $\gl (m|n)$ by Brundan \cite{brundan1} (see also \cite[Section 3.6]{germoni2}), and is implicit for $\mathfrak{osp}(2|2n)$ in work of Cheng, Wang, and Zhang \cite{CWZ}.   We provide a general proof which includes both of these as cases.  

We continue to assume that $\fg$ is a Type I Lie superalgebra which satisfies the assumptions of 
Section~\ref{SS:Kacasses}.   Given $\lambda \in X^{+}_{\0},$ let $P(\lambda)$ denote the projective cover in $\mathcal{F}$ for the simple ${\mathfrak g}$-supermodule $L(\lambda)$, and let 
\begin{align*}
Q(\lambda) = U(\fg ) \otimes_{U(\fg_{\0})} L_{\0}(\lambda) && \text{and} && K'(\lambda) = \Hom_{U(\fp^{-})}\left(U(\fg), L_{\0}(\lambda) \right),
\end{align*} where $L_{\0}(\lambda)$ is viewed as a $U(\fp^{-})$-supermodule by inflation through the canonical quotient $\fp^{-} \twoheadrightarrow \fg_{\0}$.  Note that since $L_{\0}(\lambda)$ is projective in the category of 
finite dimensional $\fg_{\0}$-supermodules (which are completely reducible over $\fg_{\0}$) 
and the functor $U(\fg ) \otimes_{U(\fg_{\0})} -$ is exact, $Q(\lambda)$ is a projective object in $\mathcal{F}.$  By Frobenius reciprocity one sees that $P(\lambda)$ appears as a direct summand of $Q(\lambda).$

Given a supermodule $M$ in $\mathcal{F},$ say $M$ admits a $K$-filtration if there is a filtration 
\[
M=M_{0} \supsetneq M_{1} \supsetneq M_{2} \supsetneq \dotsb 
\] as $\fg$-supermodules so that for each $k \geq 0$ such that $M_{k}\neq 0,$ one has $M_{k}/M_{k+1} \cong K(\mu)$ for some $\mu \in X_{\0}^{+}.$  If $M$ admits a $K$-filtration, then write $(M:K(\mu))$ for the number of times $K(\mu)$ appears.  As we will see, this number is independent of the choice of filtration.

\begin{prop}\label{L:filtrations} The following statements hold true in $\mathcal{F}.$

\begin{enumerate} 
\item [(a)]$\dim \Hom_{\fg}(K(\lambda),K'(\mu)) = \delta_{\lambda,\mu}$ for all $\lambda, \mu \in X_{\0}^{+}.$
\item  [(b)]$\operatorname{Ext}^{1}_{\mathcal{F}}(K(\lambda),K'(\mu)) = 0$ for all $\lambda, \mu \in X_{\0}^{+}.$
\item [(c)] A $\Z$-graded object of $\mathcal{F}$ has a $K$-filtration if and only if it is a $\Z$-graded free $U(\fg_{-1})$-supermodule.
\item [(d)]$P(\lambda)$ admits a $K$-filtration.
\item [(e)]If $M$ is an object of $\mathcal{F}$ with a $K$-filtration, then, 
\[
(M:K(\mu)) = \dim \Hom_{\fg}(M, K'(\mu)).
\]
\item [(f)] In particular, one has the reciprocity formula
\[
(P(\lambda):K(\mu))=[K'(\mu):L(\lambda)].
\] Furthermore, as the weights of $K'(\mu)$ are all dominated by $\mu$, one has that $(P(\lambda):K(\lambda))=1$ and $(P(\lambda):K(\mu))=0$ unless $\mu \geq \lambda.$

\end{enumerate}
\end{prop}

\begin{proof}  Parts (a) and (b) are proved in \cite[Lemma 3.6]{brundan2} and part (c) is proved in \cite[Lemma 4.2]{brundan2}.

To prove (d), consider the $\Z$-graded $\fp^{+}$-supermodule 
\[
N:=U(\fp^{+}) \otimes_{_{U(\fg_{\0})}} L_{\0}(\lambda) 
\] whose grading is obtained from the canonical $\Z$-grading on $U(\fp^{+})$ and by viewing $L_{0}(\lambda)$ as concentrated in degree $0.$  For $k \geq 0,$ set $N_{k}$ to be the $U(\fp^{+})$-supermodule generated by all elements of degree greater than or equal to $k.$  Then one has the filtration 
\[
N=N_{0} \supset N_{1} \supset \dotsb 
\] as $\fp^{+}$-supermodules, where $N_{k}/N_{k+1}$ is a finite dimensional completely reducible $\fp^{+}$-super-
module. Apply the functor $U(\fg ) \otimes_{_{U(\fp^{+})}} -$ to this filtration.   By the PBW theorem $U(\fg)$ is a free right $U(\fp^{+})$-supermodule so the functor is exact.  This along with the Tensor Identity implies that one obtains a filtration 
\[
Q(\lambda) = P_{0}\supset P_{1} \supset \dotsb 
\] where $P_{k} \cong U(\fg) \otimes_{_{U(\fp^{+})}} N_{k}.$  Furthermore, one has that 
\[
P_{k}/P_{k+1} \cong U(\fg) \otimes_{_{U(\fb^{+})}} \left(N_{k}/N_{k+1} \right) \cong \bigoplus K(\nu)
\] where the direct sum runs over all $L_{\0}(\nu)$ which appear as composition factors of the $\fp^{+}$-supermodule $N_{k}/N_{k+1}.$  Thus $Q(\lambda)$ is a $\Z$-graded supermodule in $\mathcal{F}$ which admits a $K$-filtration.  Arguing as in \cite[Theorem 3.3]{holmesnakano}, one sees that $P(\lambda)$ is a $\Z$-graded direct summand of $Q(\lambda).$  Applying (c) it follows that $P(\lambda)$ admits a $K$-filtration.

To prove (e), fix $\mu \in X_{\0}^{+}$ and say $M = M_{0} \supset M_{1} \supset \dotsb$ is a fixed $K$-filtration.  One then argues by induction on the length of the $K$-filtration of $M,$ the long exact sequence induced by the functor $\Hom_{\fg}\left(  - , K'(\mu)\right),$ and parts (a) and (b).  Finally, to prove (f), note that both sides of the equality are counted by $\dim\Hom_{\fg}(P(\lambda),K'(\mu)).$
\end{proof}

\subsection{Highest Weight Categories}\label{SS:HWC2} 

Let us now additionally assume there exists an automorphism $\tau : \fg  \to \fg$ such that 
$\tau(\fg_{i})=\fg_{-i}$ ($i \in \Z$) and, when restricted to $\fg_{0}$, $\tau$ coincides with the Chevalley 
automorphism. For Type I \emph{simple} classical Lie superalgebras such an automorphism exists 
by the proof of \cite[Proposition 2.5.3]{Kac1}. For $\gl(m|n)$ one can take $\tau$ to be the 
supertranspose map \cite[Section 4-a]{brundan1}. 

Since morphisms are not required to preserve the $\Z_{2}$-grading (cf.\ Section~\ref{SS:prelims}), $\mathcal{F}$ is not an abelian category.  However one can consider the \emph{underlying even category} $\mathcal{F}^{\text{ev}}$ which consists of the same objects as $\mathcal{F}$ but only the homomorphisms which preserve the $\Z_{2}$-grading (i.e. the even morphisms of Section~\ref{SS:prelims}).  Then $\mathcal{F}^{\text{ev}}$ is an abelian category.  Furthermore, one has the \emph{parity change functor} 
\[
\Pi : \mathcal{F}\to \mathcal{F}
\] which is defined by $\Pi M = M$ as a vector space, the $\Z_{2}$-grading given by 
\[
\left(\Pi M \right)_{r}=M_{r+1}
\] for $r \in \Z_{2},$ and action given by 
\[
x.m = (-1)^{\p{x}}xm
\] for all homogeneous $x \in \fg$ and $m \in M.$  One checks that for $M,N$ in $\mathcal{F}$,
\[
\Hom_{\mathcal{F}}(M,N)_{\1} \cong \Hom_{\mathcal{F}}(M,\Pi N)_{\0} = \Hom_{\mathcal{F}^{\text{ev}}}(M, \Pi N).
\]  Thus one can reconstruct the category $\mathcal{F}$ from $\mathcal{F}^{\text{ev}}.$  

In particular, one can choose a projective resolution $P_{\bullet} \to M$ where all maps are even and, hence, the differentials in the complex $\Hom_{\mathcal{F}}(P_{\bullet}, N)$ are degree preserving.  The chain complex then decomposes as
\[
\Hom_{\mathcal{F}}(P_{\bullet}, N) =  \Hom_{\mathcal{F}}(P_{\bullet}, N)_{\0} \oplus \Hom_{\mathcal{F}}(P_{\bullet}, N)_{\1} \cong \Hom_{\mathcal{F}^{\text{ev}}}(P_{\bullet}, N) \oplus \Hom_{\mathcal{F}^{\text{ev}}}(P_{\bullet}, \Pi N).
\]  This in turn induces the following decomposition:
\begin{equation}\label{E:extdecomp}
\operatorname{Ext}^{\bullet}_{\mathcal{F}}(M, N) =  \operatorname{Ext}^{\bullet}_{\mathcal{F}}(M, N)_{\0} \oplus \operatorname{Ext}^{\bullet}_{\mathcal{F}}(M, N)_{\1} \cong \operatorname{Ext}^{\bullet}_{\mathcal{F}^{\text{ev}}}(M, N)  \oplus \operatorname{Ext}^{\bullet}_{\mathcal{F}^{\text{ev}}}(M, \Pi N).
\end{equation}

Since $L(\lambda)$ is generated by a one dimensional highest weight space, 
$\Hom_{\fg}\left(L(\lambda),L(\lambda) \right)$ is one dimensional and elements are given by scalar multiplication.  It follows that $L(\lambda)$ and $ \Pi L(\lambda)$ are not isomorphic via a grading preserving map.  Thus a complete irredundant set of simple supermodules in $\mathcal{F}^{\text{ev}}$ is given by 
\[
\setof{L(\lambda), \Pi L(\lambda)}{\lambda \in X^{+}_{\0}}.
\]  That is, via the correspondence $(\lambda, r) \leftrightarrow \Pi^{r} L(\lambda)=:L(\lambda,r),$ one can naturally parameterize the simple supermodules in $\mathcal{F}^{\text{ev}}$ by the set 
\[
X^{++}:= X^{+}_{\0} \times \Z_{2}.
\]  Let $P(\lambda,r)=\Pi^{r}P(\lambda)$ be the projective cover of $L(\lambda,r)$ in $\mathcal{F}^{\text{ev}}$ and let $K(\lambda,r)=\Pi^{r}K(\lambda).$  Define a partial order $\leq_{\text{ev}}$ on $X^{++}$ using the lexicographic order obtained by taking the dominance order on $X^{+}_{\0}$ given by the choice of Borel subalgebra $\fb$ and the trivial partial order on $\Z_{2};$ that is, the one for which $\0$ and $\1$ are incomparable. 

Given a supermodule $M$ in the category $\mathcal{F}^{\text{ev}},$  define the \emph{contravariant dual of $M$}, $M^{\tau},$ to be $M^{*}$ with the action of $\fg$ twisted by the automorphism $\tau$ chosen above.  In particular, one sees that $(M^{\tau})^{\tau} \simeq M$ and $L(\lambda, r)^{\tau} \simeq L(\lambda,r)$ for $(\lambda,r) \in X^{++}.$  Thus, $M \mapsto M^{\tau}$ defines a duality in the sense of \cite{CPS3}.

\begin{theorem}\label{T:HWC}  The category $\mathcal{F}^{\rm{ev}}$ with duality $M \mapsto M^{\tau},$ partial order $\leq_{\rm{ev}}$ and objects $L(\lambda,r),$ $K(\lambda,r),$ and $P(\lambda,r)$ ($(\lambda,r) \in X^{++}$) is a highest weight category with duality in the sense of \cite{CPS3}.
\end{theorem}

\begin{proof}  The main thing to prove is that $P(\lambda,r)$ admits a $K$-filtration with $(P(\lambda,r):K(\mu,s)) \neq 0$ only if $(\mu,s) \geq_{\text{ev}} (\lambda,r).$  However, this follows by Proposition~\ref{L:filtrations}(f) and the definition of the partial order $\leq_{\text{ev}}$.  Using this one can now verify that $\mathcal{F}^{\text{ev}}$ is a highest weight category with respect to the partial order $\leq_{\text{ev}}.$
\end{proof}

As mentioned earlier, Brundan proved a similar result for $\gl (m|n).$  Namely, that one has the decomposition 
\[
\mathcal{F}(\gl(m|n),\gl (m|n)_{\0}) = \mathcal{F}(\gl(m|n),\gl (m|n)_{\0})^{\0} \oplus \mathcal{F}(\gl(m|n),\gl (m|n)_{\0})^{\1},
 \] where $\mathcal{F}(\gl(m|n),\gl (m|n)_{\0})^{\0}$ and $\mathcal{F}(\gl(m|n),\gl (m|n)_{\0})^{\1}$ are isomorphic highest weight categories  \cite[Theorem 4.47]{brundan1}.  One can easily show that Theorem~\ref{T:HWC} for $\fg = \gl (m|n)$ follows from Brundan's slightly stronger result.  In the case of $\fg =\mathfrak{osp}(2|2n),$ Cheng, Wang, and Zhang \cite[Section 3.1]{CWZ} use but do not explicitly prove that $\mathcal{F}$ is a highest weight category.  The above theorem includes both of these cases.

\subsection{Type II Lie Superalgebras} The work of Germoni \cite{germoni} shows that Type I Lie superalgebras are rather 
special here.  In particular, Germoni demonstrates that for the Type II Lie superalgebras $\mathfrak{osp}(3|2),$ $D(2,1;\alpha)$, 
and $G(3)$ the category $\mathcal{F}^{\rm{ev}}$ is not a highest weight category with respect to the partial order $\leq_{\text{ev}}$.  
Furthermore, examples illustrate that for Type II Lie superalgebras the Kac supermodules may have nontrivial support variety.  
This can also be deduced from \cite{germoni} but for the sake of completeness we include the calculation of an example here.  

Let $\fg=\mathfrak{osp}(3|2).$  By \cite[Sections 2.1.2 and 2.5.7]{Kac1}, $\fg$ has a consistent $\Z$-grading, 
$$ \fg = \fg_{-2} \oplus \fg_{-1} \oplus \fg_{0} \oplus \fg_{1} \oplus \fg_{2}, $$ 
but no consistent grading concentrated in degrees $-1$, $0$, and $1.$  Thus it is of Type II.   Since $\fg_{\0}$ is semisimple, the category $\mathcal{F}=\mathcal{F}(\fg,\fg_{\0})$ is simply the category of all finite dimensional $\fg$-supermodules.  Fix a Cartan subalgebra $\fh \subseteq \fg_{0}$ and Borel subalgebras  $\fb_{0} \subseteq \fg_{0}$, $\fb_{\0} \subseteq \fg_{\0}$, and $\fb \subseteq \fg$ so that $\fh \subseteq \fb_{0},$ $\fb_{\0} = \fb_{0} \oplus \fg_{2}$, and $\fb = \fb_{0} \oplus \fg_{1} \oplus \fg_{2}$.    

Let $L_0(\lambda)$ be a simple finite dimensional $\fg_0$-supermodule concentrated in degree $\0$ and of highest weight $\lambda \in \fh^*$ with respect to the choice of the pair $(\fh, \fb_0)$.  It can then be viewed as a simple supermodule for the Lie superalgebra $\mathfrak{p}:=\fg_{0} \oplus \fg_{1} \oplus \fg_{2}$ via inflation through the canonical qoutient $\mathfrak{p} \twoheadrightarrow \fg_{0}$.  Define

$$\widetilde{K}(\lambda) = U(\fg) \otimes_{_{U(\mathfrak{p})}} L_0(\lambda). $$
Using the argument used in the proof of \cite[Lemma 7.3]{brundan2} one sees that $\widetilde{K}(\lambda)$ admits a finite composition series.  Thus, it makes sense to define 
\begin{equation}\label{E:Kdef2}
K(\lambda)=\widetilde{K}(\lambda)/U
\end{equation} 
as the maximal finite dimensional quotient of $\widetilde{K}(\lambda).$  As with the Kac supermodules of Section~\ref{SS:Kacasses}, one can show that $K(\lambda)$ is the universal highest weight supermodule in $\mathcal{F}$ of highest weight $\lambda$ with respect to the choice of the pair $(\fh, \fb)$.  That is, $K(\lambda)$ is the Kac supermodule for $\fg$ of highest weight $\lambda$.  Note that $K(\lambda)$ is nonzero if and only if  $\lambda$ is the highest weight of a simple finite dimensional $\fg_{\0}$-supermodule with respect to the pair $(\fh, \fb_{\0})$.  

Let $v_{\lambda}$ denote a fixed nonzero element of the one dimensional $\lambda$ weight space of $K(\lambda)$.  By the PBW theorem, note that $K(\lambda)$ is spanned by elements of the form 
\begin{equation}\label{E:Kbasis}
x^{r}y^{s}v_{\lambda},
\end{equation} where $x \in \fg_{-1}$, $y \in \fg_{-2}$, $r \in \{0,1 \}$, and $s \geq 0$.

We now consider the case when $\lambda=0.$  Since $K(0)$ is completely reducible as a $\fg_{\0}$-supermodule and $v_{0}$ is a highest weight vector of weight $0$ with respect to the pair $(\fh, \fb_{\0})$, it follows that $v_{0}$ spans a trivial $\fg_{\0}$-supermodule in $K(0)$.  That is, $y^{s}v_{0}=0$ for all $y \in  \fg_{-2}$ and $s \geq 1$.  Furthermore, given $x \in  \fg_{-1}$ one can use \cite[Proposition 1.3]{Kacnote} (or the matrix presentation of $\fg$ provided in \cite[Section 2]{Kac1}) to choose $y \in \fg_{-2}$ and $z \in  \fg_{1}$ such that $[z,y]=x$.  But $yv_{0}=0$ by the previous remark and $zv_{0}=0$ as $K(0)$ is a quotient of $\widetilde{K}(0)$.  Thus one has  
\begin{equation}
xv_{0}=[z,y]v_{0}=zyv_{0} - yzv_{0} = 0
\end{equation} for any $x \in \fg_{-1}$.  Taken together with \eqref{E:Kbasis} this implies that $K(0)$ is one dimensional and spanned by $v_{0}.$  That is, one has 
\begin{equation}\label{E:Kzero}
K(0)=L(0)=\C.
\end{equation}
Applying the calculations of \cite[Section 8]{BKN}, one has 
$$ \mathcal{V}_{(\fg,\fg_{\0})}(K(0))=\mathcal{V}_{(\fg,\fg_{\0})}(\C)\cong \mathbb{A}^{1},$$
$$ \mathcal{V}_{(\fe,\fe_{\0})}(K(0))=\mathcal{V}_{(\fe,\fe_{\0})}(\C)\cong \mathbb{A}^{1}.$$
Thus the support varieties of a Kac module for a Type II Lie superalgebra need not be trivial. 

Another consequence of \eqref{E:Kzero} is the fact that $K(0)=L(0)$ despite the fact that $\lambda=0$ is not a typical weight.  However, in \cite[Theorem 1]{Kacnote} it is asserted that for a basic classical Lie superalgebra such as $\mathfrak{osp}(3|2),$ $K(\lambda)=L(\lambda)$ if and only if $\lambda$ is typical.  This 
counterexample may be known to experts but we could not find a suitable reference in the literature.  
We expect that the result in \cite{Kacnote} is correct for Type I Lie superalgebras.

\section{Support Varieties for the Simple Supermodules of $\gl (m|n)$}\label{S:typeA}  

\subsection{}\label{SS:introTypeA}  In this section the 
results of \cite{BKN}, \cite{dufloserganova}, and \cite{serganova3} are combined to compute the support varieties 
of the simple $\gl (m|n)$-supermodules. Before doing so, let us fix various choices and 
develop the necessary notation and background.

Let $\fg=\gl (m|n)$ be the set of all $(m+n) \times (m+n)$ matrices over the complex numbers.  
If $E_{i,j}$ ($1 \leq i,j \leq m+n$) denotes the matrix unit with a one in the $(i,j)$ 
position, then the $\Z_{2}$-grading is obtained by setting $\p{E}_{i,j}=\0$ if $1 \leq i,j \leq m$ or 
$m+1 \leq i,j \leq m+n$, and $\p{E}_{i,j}=\1$, otherwise.  The bracket is given by the super commutator,
\[
[A,B]=AB-(-1)^{\p{A}\;\p{B}}BA,
\]  for homogeneous $A,B \in \gl (m|n).$  One can fix a consistent $\Z$-grading, 
\[
\fg =\fg_{-1} \oplus \fg_{0} \oplus \fg_{1},
\] by setting $\fg_{-1}$ to be the span of $\{E_{i,j} \;\vert\; m+1 \leq i \leq m+n,\ 1 \leq j \leq m \},$  
$\fg_{0}$ to be the span of $\{E_{i,j} \;\vert\; 1 \leq i,j \leq m,\text{ or } m+1 \leq i,j \leq m+n \},$ and  $\fg_{1}$ to be the span of $\{E_{i,j} \;\vert\; 1 \leq i \leq m,\ m+1 \leq j \leq m+n \}.$  Note that $\fg$ is of Type I.  Let 
$\fh \subseteq \fg_{\0},$ $\mathfrak{b}_{\0} \subseteq \fg_{\0},$ and $\mathfrak{b} \subseteq \fg$ 
denote the subalgebras of all diagonal matrices, all upper triangular even matrices, and all upper 
triangular matrices, respectively.  

Let $G_{\0}$ denote the connected reductive group with $\operatorname{Lie}(G_{\0})=\fg_{\0}.$  Then $G_{\0} \cong GL(m) \times GL(n).$  We identify $G_{\0}$ as the subgroup of the supergroup $GL(m|n)$ embedded diagonally as block matrices (cf.\ \cite[Section 2]{kujawa}).  In particular, $G_{\0}$ acts on $\fg$ by conjugation.  Let $H\subseteq G_{\0}$ denote the maximal torus such that $\operatorname{Lie}(H)=\fh$.  Similarly, let  $B_{\0} \subseteq G_{\0}$ denote the Borel subgroup such that $\operatorname{Lie}(B_{\0})=\mathfrak{b}_{\0}.$

With respect to the choice of the pair $(\fh , \mathfrak{b} )$ the root system of $\g$ can be described as follows.  Let $\varepsilon_{i} \in \fh^{*}$ be the linear functional which picks out the $i$th entry of a diagonal matrix.  With respect to these choices we then have the roots, positive roots, even roots, and odd roots, respectively:
\begin{align*}
\Phi &= \setof{\varepsilon_{i}-\varepsilon_{j}}{1 \leq i,j \leq m+n}, \\
  \Phi^{+} &= \setof{\varepsilon_{i}-\varepsilon_{j}}{1 \leq i < j \leq m+n}, \\
\Phi_{\0} &= \setof{\varepsilon_{i}-\varepsilon_{j}}{1 \leq i,j \leq m, \text{ or } m+1 \leq i,j \leq m+n},  \\
 \Phi_{\1} &= \setof{\varepsilon_{i}-\varepsilon_{j}}{1 \leq i \leq m, m+1 \leq j \leq m+n, \text{ or }  m+1 \leq i \leq m+n, 1 \leq j \leq n}.
\end{align*}
One can define a bilinear form on $\fh^{*}$ by 
\[
(\varepsilon_{i}, \varepsilon_{j}) =\begin{cases} \delta_{i,j}, &\text{ if $1 \leq i \leq m$};\\
                                                  -\delta_{i,j}, &\text{ if $m+1 \leq i \leq m+n$}.\\
\end{cases}
\] Note that the odd roots are isotropic with respect to this bilinear form.

\subsection{Defect and Atypicality} By definition the \emph{defect} of $\fg,$ which we denote by $\operatorname{def}(\fg),$ is the maximal size of a set of pairwise orthogonal, isotropic roots of $\fg$.  One can verify that $\operatorname{def}(\fg )=\operatorname{min}(m,n)$; we write $r=\operatorname{def}(\fg )$ for short.  Given $\lambda \in X^{+}_{\0}$ 
(cf. Section~\ref{SS:Kacasses}),  the \emph{atypicality} of $\lambda,$ which we denote $\atyp (\lambda)$, is 
defined to be the maximal size of a set of pairwise orthogonal, isotropic roots which are also orthogonal to $\lambda + \rho,$ where 
\[
2\rho:=\sum_{\alpha \in \Phi^{+} \cap \Phi_{\0}} \alpha - \sum_{\alpha \in \Phi^{+} \cap \Phi_{\1}} \alpha.
\] 

By definition one has
\begin{equation}\label{E:atypvsdef}
\operatorname{atyp}(\lambda) \leq \operatorname{def}(\fg).
\end{equation}
Furthermore, note that atypicality is invariant under choice of Borel subalgebra so it makes sense to define the \emph{atypicality} of $L(\lambda)$ to be $\atyp (L(\lambda))=\atyp (\lambda).$  See \cite[Section 7]{BKN} for a more detailed discussion of defect and atypicality.  

\subsection{}\label{SS:quibbles}  When the atypicality of a simple $\gl (m|n)$-supermodule $L(\lambda)$ is zero, then necessarily $L(\lambda)$ is projective in $\mathcal{F}$ \cite[Theorem 1]{Kacnote}.  In this case $\HH^{t}(\fg,\fg_{\0}; L(\lambda)^{*} \otimes L(\lambda))=0$ for $t \geq 1$ and by \eqref{E:relsuppvardef} one has $\mathcal{V}_{(\fe ,\fe_{\0})}(L(\lambda)) = \mathcal{V}_{(\fg ,\fg_{\0})}(L(\lambda))=\{0 \}.$  Therefore we will primarily be interested in the case when the atypicality of $L(\lambda)$ is strictly greater than zero.  In this case one can assume $\lambda$ is integral up to tensoring by some one dimensional representation.  As it has no effect on support varieties by Proposition~\ref{P:translationfunctors}, we will freely and implicitly assume $\lambda$ is integral.  In particular, this is relevent to the use of the results of \cite{serganova3} in which this reduction is implicitly used.     

\subsection{Detecting Subalgebra for $\gl (m|n)$}\label{SS:eforgl}  Let 
$\mathcal{F}=\mathcal{F}(\fg, \fg_{\0})$ and, recalling the relative cohomology introduced 
in Section~\ref{SS:prelims}, let
\[
R = \HH^{\bullet}(\fg,\fg_{\0};\C),
\] 
be the cohomology ring of the category $\mathcal{F}.$  In \cite[Section 8]{BKN} we showed 
that $\fg$ admits a detecting subalgebra $\fe$ in the sense of Section~\ref{SS:detecting} and one can use 
it to calculate $R$ explicitly.  Let us sketch how this is done as it will be needed in what follows.  

As in \cite[Section 8.10]{BKN}, one can identify $\fe_{\1} \subseteq \fg_{\1}$ as the subspace spanned by the distinguished basis 
\begin{align}\label{E:ebasis}
x_{t} := E_{m+1-t,m+t} + E_{m+t,m+1-t} & &  \text{for $t=1, \dotsc, r.$}
\end{align} Let $\fe_{\0}=\operatorname{Stab}_{\fg_{\0}}(\fe_{\1}).$  Then $\fe = \fe_{\0} \oplus \fe_{\1}$ is a detecting subalgebra of $\fg .$ 

Let $W = \Z_{2}^{r} \rtimes \Sigma_{r},$ where $\Sigma_{r}$ is the symmetric group on $r$ letters and $\Sigma_{r}$ acts on $\Z_{2}^{r}$ by place permutation.  Let $X_{i} \in \fe_{\1}^{*}$ be given by $X_{i}(x_{j})=\delta_{i,j}.$  Then the symmetric algebra $S(\fe_{\1}^{*})$ is isomorphic to the polynomial ring $\C[X_{1}, \dotsc , X_{r}].$  The group $W$ naturally acts on $\C[X_{1}, \dotsc , X_{r}]$ by letting the $i$th element of $\Z_2^{r}$ act by $X_{i}\mapsto -X_{i},$ and by letting $\Sigma_{r}$ act by permuting $X_{1}, \dotsc ,X_{r}.$  By \cite[Theorem 4.4.1]{BKN} the canonical restriction map defines the following isomorphism of graded rings
\begin{align}\label{E:cohomringiso}
R  \cong \HH^{\bullet}(\fe,\fe_{\0};\C)^{W} =S(\fe_{\1}^{*})^{W}= \C[X_{1}, \dotsc , X_{r}]^{W} =\C[\dot{X}_{1}, \dotsc , \dot{X}_{r}],
\end{align}
where $\dot{X}_{1}, \dotsc , \dot{X}_{r}$ denote the elementary symmetric polynomials in the variables $X_{1}^{2}, \dotsc , X_{r}^{2}.$  The $\Z$-grading is given by $\dot{X}_{t}$ being of degree $2t,$ $t=1,\dotsc ,r.$

Note that $R$ is a polynomial ring and, in particular, has no nonzero nilpotent elements.  This point will be crucial in the proof of Proposition~\ref{P:diagrams}.

\subsection{The Case of Full Atypicality}\label{SS:fullatyp}  We first consider the special case of $\fg =\gl (k|k)$ and a simple supermodule $L(\lambda)$ of atypicality $k=\operatorname{def}(\fg)$.   
Our analysis will rely on the rank variety description of the $\fe$ support variety which was discussed in Section~\ref{SS:detecting} and a result from \cite{dufloserganova}.

\begin{prop}\label{L:fullatypicality} Let $L(\lambda)$ be a simple $\gl (k|k)$-supermodule of atypicality $k.$  Then  
\[
\mathcal{V}_{(\fa, \fa_{\0})}(L(\lambda))=\mathcal{V}_{(\fa, \fa_{\0})}(\C)
\] where $\fa$ stands for either $\fg$ or $\fe.$

\end{prop}

\begin{proof} We consider the case $\fa =\fe.$   The case $\fa =\fg$ then follows immediately from \eqref{E:resmapsbetweenvarieties} and the fact that $\operatorname{res}^{*}(\mathcal{V}_{(\fe ,\fe_{\0})}(\C))=\mathcal{V}_{(\fg ,\fg_{\0})}(\C)$ \cite[Section 6]{BKN}.  Moreover, for the purposes of the proof one can take $\fe_{\1}$ to be spanned by the matrices $E_{t,k+t}+E_{k+t,t}$ ($t=1,\dotsc ,k$) because this choice of $\fe_{\1}$ is conjugate under the Weyl group of $\fg_{\0}$ to the one given in Section~\ref{SS:eforgl}.  

Let $\mathfrak{t}$ denote the set of diagonal $k \times k$ matrices. Then one has
\[
\fe_{\1}= \left\{ x_{D}=\left(\begin{matrix} 0 & D \\
                           D & 0
\end{matrix} \right) \;\vert\; D \in \mathfrak{t}\right\}.
\]  We will prove that $L(\lambda)$ contains a trivial direct summand as an $\langle x_{D} \rangle$-supermodule where $x_{D}$ ranges over a dense open set in $\fe_{\1}$.  By the rank variety description (cf.\ Section~\ref{SS:detecting}), this will imply $x_{D} \in \mathcal{V}_{(\fe,\fe_{\0})}(L(\lambda))$ for all $x_{D}$ in this open set.  This then implies the result as the set of all such $x_{D}$ is dense in $\mathcal{V}_{(\fe,\fe_{\0})}(\C)$ and $\mathcal{V}_{(\fe,\fe_{\0})}(L(\lambda))$ is a closed subset of $\mathcal{V}_{(\fe,\fe_{\0})}(\C).$

By the structure of indecomposable $\langle x_{D} \rangle$-supermodules provided in \cite[Proposition 5.2.1]{BKN}, to prove that $L(\lambda)$ has a trivial $\langle x_{D}\rangle$ direct summand it suffices to prove the following two claims.

\medskip
 \label{claim}
\noindent \emph{Claim 1: There exists a vector $v \in L(\lambda)$ such that $x_{D}v=0.$}

\noindent \emph{Claim 2: There does not exist a vector $w \in L(\lambda)$ such that $x_{D}w=v.$}
\medskip

We first require some preliminaries.  Let $T$ denote the set of invertible $k \times k$ matrices.  Throughout, let $D$ be an element of $T$ and let $I_{k}$ denote the $k \times k$ identity matrix.  Set

\[
g_{D}=\left(\begin{matrix} I_{k} & 0 \\
                           0 & D^{-1}
\end{matrix} \right) \in G_{\0}.
\] 
Furthermore, let $x_{D}^{+}, x_{D}^{-} \in \fg$ be given by
\[
x^{+}_{D}=\left(\begin{matrix} 0 & D \\
                           0 & 0
\end{matrix} \right) \hspace{.5in} \text{and} \hspace{.5in} x^{-}_{D}=\left(\begin{matrix} 0 & 0 \\
                           D & 0
\end{matrix} \right).
\] Then $x_{D}^{+} \in \fg_{1},$ $x_{D}^{-}\in \fg_{-1},$ and $x_{D}=x_{D}^{-} + x_{D}^{+}.$

Let $\varphi_{D}: G_{\0} \to G_{\0}$ be the automorphism given by conjugation by $g_{D};$ that is,
\[
h \mapsto g_{D}hg_{D}^{-1}.
\]  Let $K$ denote the image in $G_{\0}$ of $GL(k)$ under the diagonal embedding and set
\begin{equation}\label{E:twistingK}
K^{+}_{D}=\varphi_{D}(K)  \hspace{.5in}\text{and}\hspace{.5in}  K^{-}_{D}=\varphi_{D^{-1}}(K).
\end{equation}
One can check that 
\begin{equation}\label{E:KDdef}
\operatorname{Stab}_{G_{\0}}\left(x^{\pm}_{D} \right) = K_{D}^{\pm}.
\end{equation}

If $M$ is a $G_{\0}$-module, then let $M^{\varphi_{D}}$ denote the twist of $M$ by $\varphi_{D}$.  The map $\alpha_{D}:M \to M^{\varphi_{D}}$ given by $m \mapsto g_{D}m$ is then an isomorphism of $G_{\0}$-modules. In particular, $M$ and $M^{\varphi_{D}}$ are also isomorphic as $K$-modules.  That is, by \eqref{E:twistingK} $M$ has the same structure as a module for $K,$ $K_{D}^{+},$ and $K^{-}_{D}$.  In particular, $M$ will have a trivial $K^{\pm}_{D}$-submodule if and only if it has a trivial $K$-submodule.  Applying this observation and results on the $K$-module structure of $L(\lambda)$ in \cite[Lemma 10.4]{dufloserganova}, one obtains the following key fact.

\medskip
\noindent $(*)$ \label{fact}  \emph{The supermodule $L(\lambda)$ is a $\mathbb{Z}$-graded $\fg$-supermodule, 
\[
L(\lambda) = L_{0} \oplus L_{-1} \oplus L_{-2} \oplus \dotsb,
\] where the above direct sum decomposition is as $G_{\0}$-modules.  In this decomposition $L_{0}$ contains a trivial $K$-submodule whereas $L_{-1}$ does not contain a trivial $K$-submodule.  Consequently, $L_{0}$ contains a trivial $K^{\pm}_{D}$-submodule whereas $L_{-1}$ does not contain a trivial $K^{\pm}_{D}$-submodule. } 
\medskip

Furthermore, by \cite[Lemma 10.4]{dufloserganova} one has that $L_{0} \cong S \boxtimes S^{*}$ as a $G_{\0}\cong GL(k) \times GL(k)$-module, where $S$ denotes a simple $GL(k)$-module. Hence,
\begin{equation}\label{E:Lzeroiso}
 L_{0} \cong S\otimes S^{*}
\end{equation}
as a $K$-module.  Thus there is a nontrivial $K$-linear map 
\[
\tau: L_{0} \to \C
\] given by $s \otimes f \mapsto f(s).$  If one fixes a nonzero vector $v_{0} \in L_{0}$ which spans a trivial $K$-module then by \eqref{E:Lzeroiso} one has 
\begin{equation}\label{E:tauonvzero}
\tau(v_{0}) \neq 0.
\end{equation}  More generally, if $\rho : GL(k) \to GL(S)$ is the representation corresponding to the $GL(k)$-module $S$ and 
\[
g=\left(\begin{matrix} A & 0 \\
                           0 & B
\end{matrix} \right) \in G_{\0},
\] then 
\begin{equation}\label{E:tauong}
\tau(gv_{0}) = \frac{\operatorname{trace}(\rho(B^{-1}A))}{\dim S}\tau(v_{0}).
\end{equation}  The above equality can be verified by fixing a basis for $S,$ the corresponding dual basis for $S^{*},$ and using them to compute both sides of \eqref{E:tauong}.

In addition, one can define a $K$-linear map 
\[
\widetilde{\tau}: L_{-1} \to \C 
\] by $m \mapsto \tau(x_{I_{k}}^{+}m).$  Note that by the $\Z$-grading, if $m\in L_{-1},$ then $x_{I_{k}}^{+}m \in L_{0}$ so it makes sense to apply $\tau.$ Also note that $K$-linearity follows from \eqref{E:KDdef} (applied to the case $D=I_{k}$).  However, by $(*)$ one knows that $L_{-1}$ does not contain any trivial $K$-modules; hence,
\begin{equation}\label{E:tildetauiszero}
\widetilde{\tau} =0.
\end{equation}

We are now prepared to prove Claims $1$ and $2$.  First, consider Claim $1.$  By \eqref{E:twistingK} one has that
\begin{equation}\label{E:vdef}
v:=\alpha_{D^{-1}}(v_{0})=g_{D^{-1}}v_{0}=g_{D}^{-1}v_{0}
\end{equation}
spans a trivial $K_{D}^{-}$-module in $L_{0}.$ By the $\mathbb{Z}$-grading on $L(\lambda)$ one has that $x_{D}v=x_{D}^{-}v.$  If $x_{D}^{-}v \neq 0,$ then this would imply by \eqref{E:KDdef} that there is a trivial $K_{D}^{-}$-module in $L_{-1}$ spanned by $x_{D}^{-}v,$ contradicting $(*)$.  Therefore $x_{D}v=0.$

To prove Claim $2,$ assume that there exists $w \in L$ such that $x_{D}w=v.$  If one uses the $\Z$-grading in $(*)$ and writes $w=w_{0} + w_{-1}+ \dotsb + w_{t}$ so that $w_{i} \in L_{i}$ ($i=0, -1, -2, \dotsc) ,$ then one sees by the $\Z$-grading that $x_{D}w=v$ only if 
\begin{equation}\label{E:wequals}
x_{D}^{+}w_{-1}=v.
\end{equation}

Recall that $G_{\0}$ acts on $\fg$ by conjugation and that this is compatible with the action of $\fg$ on $L(\lambda)$ in the sense that $g(xm)=(gxg^{-1})(gm)$ for all $g \in G_{\0},$ $x \in \fg ,$ and $m \in L(\lambda).$  Therefore, using \eqref{E:wequals}, \eqref{E:vdef}, and the fact that $x_{D}^{+}=g_{D}x_{I_{k}}^{+}g^{-1}_{D}$, one has $x_{D}w=v$ only if 
\begin{equation}\label{E:keyequals}
x_{I_{k}}^{+}\left( g^{-1}_{D}w_{-1}\right)= g_{D}^{-1}v=g_{D}^{-2}v_{0}.
\end{equation}
Applying $\tau$ to both sides of \eqref{E:keyequals} and using that $\widetilde{\tau}=0$ by \eqref{E:tildetauiszero}, yields
\[
0=\widetilde{\tau}\left(g_{D}^{-1}w_{-1} \right) = \tau\left(g_{D}^{-2}v_{0} \right).
\]  However, applying \eqref{E:tauong} and the fact that $\tau(v_{0})\neq 0$, one obtains
\[
0 = \operatorname{trace}\left(\rho (D^{-2}) \right)=\chi_{S}(D^{-2}),
\] where by definition $\chi_{S}=(\operatorname{trace} \circ \rho) : T \to \C$ is the character defined on the torus $T \subseteq GL(k)$ by the representation $\rho.$  In short, one concludes that if $x_{D}w=v$ for some $w \in L(\lambda),$ then 
\[
\chi_{S}(D^{-2})=0.
\]

However, the map $T \to \C$ given by $D \mapsto \chi_{S}(D^{-2})$ is continuous and not identically zero so there is a dense open set $\mathcal{O}\subseteq T \subseteq \mathfrak{t}$ for which $\chi_{S}(D^{-2})\neq 0.$  Thus, for $D \in \mathcal{O}$ there does not exist $w \in L(\lambda)$ for which $x_{D}w=v.$  That is, since the map $\mathfrak{t}\to \fe_{\1}$ given by $D \mapsto x_{D}$ is a homeomorphism, $\left\{x_{D} \;\vert\; D \in \mathcal{O} \right\}$ is a dense open set of $\fe_{\1}$ for which Claims $1$ and $2$ are both satisfied.  As was explained at the beginning of the proof, this suffices to prove the theorem.
\end{proof}

\subsection{Translation Functors}\label{SS:translationfunctors}  By, for example, \cite[Section 3]{serganova3} the category $\mathcal{F}$ decomposes into blocks given by central characters.  Thus one has 
\begin{equation}\label{E:decompofF}
\mathcal{F} = \bigoplus_{\chi} \mathcal{F}_{\chi},
\end{equation}
where the direct sum runs over the central characters of $U(\fg)$ and where $\mathcal{F}_{\chi}$ denotes the block corresponding to the character $\chi.$  Furthermore, since each simple supermodule in a given fixed block has the same atypicality (cf.\ \cite[Corollary 3.3]{serganova3}), it makes sense to define the atypicality of a block to be the atypicality of a simple supermodule in the block.  The first main result of \cite{serganova3} is the following theorem.

\begin{theorem}\label{T:atypicalityequivalence}\cite[Theorem 3.5]{serganova3}  The blocks $\mathcal{F}_{\chi}$ and $\mathcal{F}_{\chi'}$ are equivalent as categories if and only if they have the same atypicality.
\end{theorem}

Let us mention how one constructs the functors which provide this equivalence of categories.  Let $\operatorname{pr}_{\chi}: \mathcal{F} \to \mathcal{F}_{\chi}$ denote the canonical projection functor given by the direct sum decomposition \eqref{E:decompofF}. Letting $E$ denote a simple supermodule in $\mathcal{F},$ one can define \emph{translation functors} 
\begin{equation}\label{E:transfunctorsdef}
T_{E}^{\chi,\chi'}: \mathcal{F}_{\chi} \to \mathcal{F}_{\chi'}
\end{equation}
by  
\[
M \mapsto \operatorname{pr}_{\chi'}(M\otimes E),
\] for $M$ an object of $\mathcal{F}_{\chi}.$  When $\chi$ and $\chi'$ are ``adjacent'' in a sense made precise in \cite[Lemma 5.5]{serganova3}, one can make a suitable choice of simple supermodule $E$ so that the translation functors $T_{E}^{\chi,\chi'}$ and $T_{E^{*}}^{\chi',\chi}$ provide the equivalence between neighboring blocks.  The composition of such equivalences yields Theorem~\ref{T:atypicalityequivalence}.

In the next proposition we will see that the support varieties of \cite{BKN} behave well with respect to these translation functors.  Before doing so, recall the following facts about support varieties.  Let $\fa$ denote either $\fg$ or $\fe$ and let $M, N$ be objects in $\mathcal{F}(\fa,\fa_{\0}).$  Then by the argument from finite groups \cite[Proposition 5.7.5]{benson} one has 
\begin{equation}\label{E:directsum}
\mathcal{V}_{(\fa, \fa_{\0})}(M \oplus N) = \mathcal{V}_{(\fa, \fa_{\0})}(M) \cup \mathcal{V}_{(\fa, \fa_{\0})}(N).
\end{equation}  Using the fact that as functors one has  
\[
- \otimes (M \otimes N) \cong  \left(- \otimes N \right) \circ \left(- \otimes M \right) \cong  \left(- \otimes M \right) \circ \left(- \otimes N \right),
\] it follows that
\begin{equation}\label{E:tensorproduct}
\mathcal{V}_{(\fa, \fa_{\0})}(M \otimes N) \subseteq \mathcal{V}_{(\fa,\fa_{\0})}(M) \cap 
\mathcal{V}_{(\fa,\fa_{\0})}(N).
\end{equation}

\begin{prop}\label{P:translationfunctors}  Let $\chi,$ $\chi',$ and $E$ be such that  $T_{E}^{\chi,\chi'}$ and $T_{E^{*}}^{\chi',\chi}$ provide an equivalence of categories between $\mathcal{F}_{\chi}$ and $\mathcal{F}_{\chi'}.$  If $M$ (resp.\ $N$) is an object of $\mathcal{F}_{\chi}$ (resp.\ $\mathcal{F}_{\chi'}$), then 
\begin{align*}
\mathcal{V}_{(\fa, \fa_{\0})}\left( T_{E}^{\chi,\chi'}M\right) = \mathcal{V}_{(\fa, \fa_{\0})}(M) & & \text{and} & & \mathcal{V}_{(\fa, \fa_{\0})}\left( T_{E^{*}}^{\chi',\chi}N\right) = \mathcal{V}_{(\fa, \fa_{\0})}(N).
\end{align*}
Here $\fa$ denotes either $\fg$ or $\fe .$  
\end{prop}

\begin{proof}  Taking \eqref{E:directsum} and~\eqref{E:tensorproduct} together, if $M$ is an object of $\mathcal{F}_{\chi},$ then one has 
\begin{equation}\label{E:transfunctorsonvarieties}
\mathcal{V}_{(\fa, \fa_{\0})}\left( T_{E}^{\chi,\chi'}M\right) \subseteq \mathcal{V}_{(\fa, \fa_{\0})}(M \otimes E) \subseteq \mathcal{V}_{(\fa, \fa_{\0})}(M) \cap \mathcal{V}_{(\fa, \fa_{\0})}(E) \subseteq \mathcal{V}_{(\fa, \fa_{\0})}(M).
\end{equation}  Similarly, one has 
\[
\mathcal{V}_{(\fa, \fa_{\0})}\left( T_{E^{*}}^{\chi',\chi}N\right) \subseteq \mathcal{V}_{(\fa, \fa_{\0})}(N \otimes E^{*}) 
\subseteq \mathcal{V}_{(\fa, \fa_{\0})}(N) \cap \mathcal{V}_{(\fa, \fa_{\0})}(E^{*}) \subseteq \mathcal{V}_{(\fa, \fa_{\0})}(N).
\]  
Since $ T_{E^{*}}^{\chi',\chi} \left( T_{E}^{\chi,\chi'}M\right) \cong M$ 
one can apply  \eqref{E:transfunctorsonvarieties}, to obtain  
\begin{align*}
\mathcal{V}_{(\fa, \fa_{\0})}(M) &= \mathcal{V}_{(\fa, \fa_{\0})}\left( T_{E^{*}}^{\chi',\chi} \left( T_{E}^{\chi,\chi'}M\right)\right)  \subseteq \mathcal{V}_{(\fa, \fa_{\0})}\left( T_{E}^{\chi,\chi'}M\right) \subseteq \mathcal{V}_{(\fa, \fa_{\0})}(M).
\end{align*}  The first equality follows.  The second equality follows by switching $\chi$ and $\chi',$ $E$ and $E^{*}$, and $M$ and $N$.
\end{proof}

\subsection{Restriction Functors}\label{SS:resindfunctors}  For clarity and brevity, set $\mathcal{F}(m|n)=\mathcal{F}(\gl(m|n),\gl(m|n)_{\0}).$
Denote the block containing the trivial supermodule by $\mathcal{F}(m|n)_{\chi_{\C}}$.  The second main result of Serganova \cite[Theorem 3.6]{serganova3} is the following theorem.

\begin{theorem} \label{T:equivalencetoprincipalblock}  Let $\mathcal{F}(m|n)_{\chi}$ be a block of atypicality $k.$  Then  $\mathcal{F}(m|n)_{\chi}$ and  $\mathcal{F}(k|k)_{\chi_{\C}}$ are equivalent categories.
\end{theorem}

The equivalence of these two categories is provided by a functor 
\[
\Phi:  \mathcal{F}(m|n)_{\chi} \to \mathcal{F}(k|k)_{\chi_{\C}} 
\]
which, very roughly speaking, is of the form 
\begin{equation}\label{E:Phidef}
\Res_{\mu}  \circ T_{t} \circ \dotsb \circ T_{1} 
\end{equation}
where the functors $T_{i}$ are translation functors which give equivalances between blocks of $\mathcal F(m|n)$ as in Theorem~\ref{T:atypicalityequivalence}, and where $\Res_{\mu}$ is a functor which refines restriction.  By Proposition~\ref{P:translationfunctors} the $\fg$ and $\fe$ support varieties are invariant under the functors $T_{i}.$  Thus the main task is to understand how they behave with respect to the functor $\Res_{\mu}.$

First, let us define $\Res_{\mu}.$  Fix a block $\mathcal{F}(m|n)_{\chi}$ of atypicality $k$ and note that by \eqref{E:atypvsdef} one has $k \leq m,n.$  Consequently we can naturally identify $\gl(k|k)$ as the subalgebra of $\gl(m|n)$ spanned by $\{ E_{i,j} \;\vert\; m-k+1 \leq i,j \leq m+k \}.$  Under this identification the Cartan subalgebra $\mathfrak{h}_{m|n} \subseteq \gl(m|n)$ naturally decomposes,
\[
\mathfrak{h}_{m|n} = \mathfrak{h}_{k|k} \oplus \mathfrak{h}',
\] where $\mathfrak{h}'$ is the linear span of $E_{t,t}$ ($t=1, \dotsc ,m-k, m+k+1, \dotsc , m+n$).  Dualizing this decomposition one obtains,
\[
\mathfrak{h}_{m|n}^{*} = \mathfrak{h}_{k|k}^{*} \oplus (\mathfrak{h}')^{*}.
\]

As in Section~\ref{SS:eforgl}, we identify $\fe_{\1} \subseteq \fg_{\1}$ as the subspace spanned by the distinguished basis 
\begin{align*}
x_{t} := E_{m+1-t,m+t} + E_{m+t,m+1-t} & &  \text{for $t=1, \dotsc, r.$}
\end{align*}  Note that this is compatible with the subalgebra $\gl (k|k)$ in the sense that if $\widetilde{e}$ is the detecting subalgebra for $\gl (k|k)$, then $\widetilde{\fe}_{\1}\subseteq \fe_{\1}$ and $\widetilde{\fe}_{\1}$ is spanned by the elements $x_{1}, \dotsc , x_{k}.$  Furthermore, one then has $\widetilde{W} \subseteq W$ as the subgroup of elements which fix $X_{k+1}, \dotsc , X_{r}.$  Here $X_{j}$ ($j=1, \dotsc, r$) is as in Section~\ref{SS:eforgl}.

Fix a block $\mathcal{F}(m|n)_{\chi}$ of atypicality $k$ and $\mu \in (\mathfrak{h}')^{*}.$  Let $\mathcal{F}(m|n)_{\chi}^{\mu}$ denote the full subcategory of $\mathcal{F}(m|n)_{\chi}$ consisting of all supermodules whose simple composition factors are isomorphic to $L(\lambda + \mu)$ for some $\lambda \in  \mathfrak{h}_{k|k}^{*}.$  Let  $\mathcal{F}(k|k)_{\chi_{\C}}^{\mu}$ denote the full subcategory of  $\mathcal{F}(k|k)_{\chi_{\C}}$ of all supermodules whose simple composition factors are isomorphic to $L(\lambda)$ for some $\lambda \in  \mathfrak{h}_{k|k}^{*}$ such that $L(\lambda + \mu)$ is an object of $\mathcal{F}(m|n)$ (i.e., $\lambda+\mu \in X^{+}_{\0}$).

One can then define the refined restriction functor
\[
\Res_{\mu}:  \mathcal{F}(m|n)_{\chi}^{\mu} \to \mathcal{F}(k|k)_{\chi_{\C}}^{\mu}
\] by 
\[
\Res_{\mu}M = \setof{x \in M}{hx=\mu(h)x \text{ for all } h \in \mathfrak{h}'},
\] viewed as a $\gl (k|k)$-supermodule by restriction.  Note that $\fh'$ and $\gl (k|k)$ are commuting subalgebras of $\gl (m|n)$ so $\Res_{\mu}M$ is in fact a $\gl (k|k)$-supermodule.  A key lemma is the following result from \cite[Lemma 6.3]{serganova3}.

\begin{prop}\label{P:serganovaskeylemma}  Let $\mathcal{F}(m|n)_{\chi}$ be a block of atypicality $k$ and $\mu \in (\fh')^{*}.$  Let $\mathcal{F}(m|n)_{\chi}^{\mu}$ and $\mathcal{F}(k|k)_{\chi_{\C}}^{\mu}$ be as above. Then the functor 
\[
\Res_{\mu}:  \mathcal{F}(m|n)_{\chi}^{\mu} \to \mathcal{F}(k|k)_{\chi_{\C}}^{\mu}
\] defines an equivalence of categories.
\end{prop}

The functor $\Phi$ alluded to in \eqref{E:Phidef} is then defined as follows.  First one fixes a certain specific block $\mathcal{F}(m|n)_{\chi_{0}}$ of atypicality $k$ (recalling that they are all equivalent by Theorem~\ref{T:atypicalityequivalence}).  One chooses a suitable sequence of elements $\mu_{1}, \mu_{2}, \dotsc  \in (\mathfrak{h}')^{*};$ a sequence of central characters $\chi_{1}, \chi_{2}, \dotsc$ ; and a sequence of translation functors  $T_{1}, T_{2}, \dotsc $ so that $T_{i}: \mathcal{F}(m|n)_{\chi_{i-1}} \to  \mathcal{F}(m|n)_{\chi_{i}}$ is an equivalence between blocks of atypicality $k$. The functor $\Phi$ is then defined by 
\begin{equation}\label{E:Phidef2}
\Phi = \lim_{i \to \infty} \Res_{\mu_{i}} \circ T_{i} \circ \dotsb \circ T_{1}.
\end{equation}
Note that one has to verify that the limit makes sense by verifying that for $M \in \mathcal{F}_{\chi_{0}}$ one has $\left( T_{i} \circ \dotsb \circ T_{1} \right)M  \in \mathcal{F}(m|n)_{\chi_{i}}^{\mu_{i}}$ for $i \gg 0,$ and by verifying the appropriate compatibility condition.  This is done in \cite[Lemma 6.4]{serganova3} along with a more precise description of the functor $\Phi$ (e.g.\ the choice of $\mathcal{F}(m|n)_{\chi_{0}}$ and the sequences $\mu_{1}, \mu_{2}, \dotsc;$ $\chi_{1}, \chi_{2}, \dotsc ;$ and $T_{1},T_{2}, \dotsc$).

We now turn to understanding the precise relationship between the functor $\Res_{\mu}$ and support varieties.  We continue to use the notation for relative cohomology introduced in Section~\ref{SS:morecohom}.  The inclusion $\gl (k|k) \hookrightarrow \gl (m|n)$ induces a map 
\[
\res : \HH^{\bullet}(\gl(m|n), \gl (m|n)_{\0 }; M) \to  \HH^{\bullet}(\gl(k|k), \gl (k|k)_{\0 }; M)
\] for any $M$ in $\mathcal{F}(m|n).$  Note that this coincides with the map induced by the restriction functor, $\Res :\mathcal{F}(m|n) \to \mathcal{F}(k|k).$  We then have the following commutative diagram. 

\begin{equation}\label{E:commute}
\begin{CD}
 I_{m|n}(M) \hookrightarrow \HH^{\bullet}(\gl(m|n), \gl (m|n)_{\0 }; \C)  @>^{m_{1}}>>        \HH^{\bullet}(\gl(m|n), \gl (m|n)_{\0 }; M\otimes M^{*}) \\
@V\res_{\C} VV                                                           @VV\res V\\
  I_{k|k}(M) \hookrightarrow \HH^{\bullet}(\gl(k|k), \gl (k|k)_{\0 }; \C)  @>^{m_{2}}>>        \HH^{\bullet}(\gl(k|k), \gl (k|k)_{\0 }; M\otimes M^{*}) 
\end{CD}
\end{equation}
Here the horizontal maps are those induced by the functor $-\otimes M$, and $I_{m|n}(M)$ (resp.\ $I_{k|k}(M)$) is the kernel of this map; that is, as with finite groups \cite{benson}, it is an ideal whose zero set is $\mathcal{V}_{(\gl (m|n), \gl (m|n)_{\0})}(M)$ (resp.\  $\mathcal{V}_{(\gl(k|k), \gl (k|k)_{\0})}(M)$).  

Also, for the purposes of uniformity in our notation we shall capitalize the names of functors and when they induce maps in cohomology call the induced maps by the same name but in lower case.  For example, if $F: \mathcal{C}_{1} \to \mathcal{C}_{2}$ is an exact functor between categories of supermodules for Lie superalgebras, then we write 
\[
f: \Ext_{\mathcal{C}_{1}}^{\bullet}\left(M, N \right) \to \Ext_{\mathcal{C}_{2}}^{\bullet}\left(FM, FN \right)
\]  for the induced map of $\Z$-graded superspaces.  Thus, in \eqref{E:commute} $\res_{\C}$ denotes the map induced by the exact functor $\Res$ (with coefficients in the trivial supermodule).

Let $\chi_{0}$ be the fixed central character of atypicality $k,$ $\mu_{1}, \mu_{2}, \dotsc$ the elements of $(\fh')^{*},$ $\chi_{1}, \chi_{2} \dotsc$ the central characters, and $T_{1}, T_{2}, \dotsc$ the sequence of translation functors, all chosen as discussed above.  The following proposition records certain properties of \eqref{E:commute}.

\begin{prop}\label{P:diagrams}    Let $J$ denote the kernel of the map $\res_{\C}$ and fix $ d \geq 0$ so that $J$ is generated by elements of degree no more than $d.$  Then the following statements about \eqref{E:commute} hold true.
\begin{enumerate}
\item [(a)] The map $\res_{\C}$ is a surjective algebra homomorphism.
\item  [(b)] Let $M=L(\lambda)$ be a simple supermodule in a block $\mathcal{F}(m|n)_{\chi}$ of atypicality $k.$  Then the map $m_{2}$ is injective. 
\item  [(c)] Assume $M$ is an object of $\mathcal{F}(m|n)_{\chi}^{\mu}$ for some block $\mathcal{F}(m|n)_{\chi}$ of atypicality $k$ and some $\mu \in (\fh')^{*}.$ Further assume that 
\[
 \Ext^{i}_{\mathcal{F}(m|n)_{\chi}}(M,M) = \Ext^{i}_{\mathcal{F}(m|n)_{\chi}^{\mu}}(M,M)
\]  for $i=0, \dotsc, d.$ Then $J \subseteq I_{m|n}(M).$
\end{enumerate}
\end{prop}

\begin{proof} 

\noindent  One proves (a) as follows.  Let us write $\fe \subseteq \gl (m|n)$ and $\tilde{\fe} \subseteq \gl (k|k)$ for the detecting subalgebras given in Section~\ref{SS:eforgl}.   By \cite[Theorem 3.3.1(a)]{BKN} (see also \eqref{E:cohomringiso}), restriction induces isomorphisms  $ \HH^{\bullet}(\gl(m|n), \gl (m|n)_{\0 }; \C) \to S(\fe_{\1}^{*})^{\mathcal{W}}$ and $ \HH^{\bullet}(\gl(k|k), \gl (k|k)_{\0 }; \C)  \to  S(\tilde{\fe}_{\1}^{*})^{\widetilde{\mathcal{W}}}.$

{}From the identification $\widetilde{\fe}_{\1}\subseteq \fe_{\1},$ one has the canonical homomorphism $S(\fe^{*}_{\1}) \to S(\tilde{\fe}^{*}_{\1})$ given by restriction of functions.  From the explicit description given in Section~\ref{SS:eforgl} one sees that this map corresponds to setting $X_{k+1},\dotsc ,X_{r}$ to zero.   Restriction of this map in turn induces a map 
\[
\rho: S(\fe^{*}_{\1})^{\mathcal{W}} \to S(\tilde{\fe}^{*}_{\1})^{\widetilde{\mathcal{W}}}.
\]  The explicit description given in \eqref{E:cohomringiso} allows one to verify that this map is surjective.

As all maps are induced by restrictions, one has the following commutative diagram.
\begin{equation*}
\begin{CD}
 \HH^{\bullet}(\gl(m|n), \gl (m|n)_{\0 }; \C)  @>_{\simeq}>>       S(\fe_{\1}^{*})^{\mathcal{W}} \\
@V\res_{\C} VV                                                           @VV \rho  V\\
  \HH^{\bullet}(\gl(k|k), \gl (k|k)_{\0 }; \C)  @>_{\simeq}>>         S(\tilde{\fe}_{\1}^{*})^{\widetilde{\mathcal{W}}} 
\end{CD}
\end{equation*}  Therefore, the map $\res_{\C}$ is surjective.

To prove (b) one argues as follows.  Fix $\mu \in (\fh')^{*}$ so that $L(\lambda)$ lies in $\mathcal{F}(m|n)^{\mu}_{\chi}.$  For objects in $\mathcal{F}(m|n)_{\chi}^{\mu}$ one then has the decomposition 
\[
\Res(M) = \Res_{\mu}(M) \oplus G_{\mu}(M)
\] where the functor $G_{\mu}: \mathcal{F}(m|n)_{\chi}^{\mu} \to \mathcal{F}(k|k)$ is defined by 
\begin{equation}\label{E:Gmudef}
G_{\mu}(N)=\sum_{\substack{\nu \in (\mathfrak{h}')^{*}\\ \nu \neq \mu}}\setof{x \in N}{hx = \nu(h)x \text{ for all } h \in \mathfrak{h}'}. 
\end{equation}    By the additivity of the bifunctor $\HH^{\bullet}(\gl(k|k), \gl (k|k)_{\0 }; -\otimes -)$ it follows that
\begin{align}\label{E:bifunctordecomp}
 \HH^{\bullet}(\gl(k|k), \gl (k|k)_{\0 }; &L(\lambda)\otimes L(\lambda)^{*}) \\
   & \cong  \HH^{\bullet}(\gl(k|k), \gl (k|k)_{\0 }; \Res_{\mu}(L(\lambda))\otimes \Res_{\mu}(L(\lambda))^{*}) \oplus (**), \notag 
\end{align} where $(**)$ denotes the appropriate complementary superspace.
By Proposition~\ref{P:serganovaskeylemma}, $\Res_{\mu}(L(\lambda))$ is a simple supermodule of atypicality $k$ in $\mathcal{F}(k|k).$ Hence by Proposition~\ref{L:fullatypicality} one has that $\mathcal{V}_{\gl (k|k)}\left( \Res_{\mu}(L(\lambda))\right)=\mathcal{V}_{\gl (k|k)}\left(\C\right).$  That is, since $\HH^{\bullet}(\gl(k|k), \gl (k|k)_{\0 }; \C)$ has no nonzero nilpotent elements,
\[
\operatorname{Ann}_{\HH^{\bullet}(\gl(k|k), \gl (k|k)_{\0 }; \C)}\left( \HH^{\bullet}(\gl(k|k), \gl (k|k)_{\0 }; \Res_{\mu}(L(\lambda))\otimes \Res_{\mu}(L(\lambda))^{*})\right) = (0).
\]  Consequently,
\[
 I_{k|k}(L(\lambda))=\operatorname{Ann}_{\HH^{\bullet}(\gl(k|k), \gl(k|k)_{\0}; \C)}\left( \HH^{\bullet}(\gl(k|k), \gl (k|k)_{\0 }; L(\lambda) \otimes L(\lambda)^{*})\right)=(0) 
\]   by \eqref{E:bifunctordecomp}.  The result then follows.

 To prove $(c),$ we first prove the following claim.

\medskip
\noindent $(*)$ \label{claim3} \emph{The map 
\[
\res: \Ext^{i}_{\mathcal{F}(m|n)_{\chi}}(M,M) \to \Ext^{i}_{\mathcal{F}(k|k)}(M,M) 
\] is injective for $i=0,\dotsc,d$.}
\medskip

By assumption there is a $\mu \in (\fh ')^{*}$ such that, 
\[
\Ext^{i}_{\mathcal{F}(m|n)^{\mu}_{\chi}}(M,M) = \Ext^{i}_{\mathcal{F}(m|n)_{\chi}}(M,M)
\] for $i=0, \dotsc , d.$  Thus it suffices to show 
\[
\res: \Ext^{i}_{\mathcal{F}(m|n)^{\mu}_{\chi}}(M,M) \to \Ext^{i}_{\mathcal{F}(k|k)}(M,M) 
\]  is injective for $i=1, \dotsc ,d.$

To prove this, let $\mathcal{C} \subseteq \mathcal{F}(k|k)$ denote the full subcategory consisting of all objects of $\mathcal{F}(k|k)$ which are $\gl (k|k) \oplus \fh'$-supermodules and semisimple as $\fh'$-supermodules.  One then has a restriction functor $\Res : \mathcal{F}(m|n) \to \mathcal{C}$ which, when composed with the functor which forgets the action of $\fh',$ yields the functor $\Res: \mathcal{F}(m|n) \to \mathcal{F}(k|k).$   Let $P_{\mu}: \mathcal{C} \to \mathcal{F}(k|k)$ be the functor given by 
\[
P_{\mu}(N)= \setof{x \in N}{hx=\mu(h)x \text{ for all } h \in \fh'}.
\]  Since $\fh '$ and $\gl (k|k)$ are commuting superalgebras, this is a supermodule for $\gl(k|k).$  One then has the following factorization of $\Res_{\mu}:$  
\begin{equation}\label{E:functordecomp}
\Res_{\mu} =  P_{\mu} \circ \Res.
\end{equation}

Now, since $\Res_{\mu}: \mathcal{F}(k|k)_{\chi}^{\mu} \to \mathcal{F}(k|k)_{\chi_{\C}}^{\mu}$ is an equivalence of categories by Proposition~\ref{P:serganovaskeylemma}, it follows that the induced map 
\[
\res_{\mu}: \Ext^{i}_{\mathcal{F}(m|n)_{\chi}^{\mu}}(M,M) \to \Ext^{i}_{\mathcal{F}(k|k)_{\chi_{\C}}^{\mu}}(M,M) \hookrightarrow \Ext^{i}_{\mathcal{F}(k|k)}(M,M)
\] is injective for all $i \geq 0$.  But \eqref{E:functordecomp} implies the induced maps satisfy $\res_{\mu} = p_{\mu} \circ \res.$  This along with the fact that the forgetful functor $\mathcal{C} \to \mathcal{F}(k|k)$ induces an injective map in cohomology implies the map
\[
\res: \Ext^{i}_{\mathcal{F}(m|n)_{\chi}^{\mu}}(M,M) \to \Ext^{i}_{\mathcal{F}(k|k)}(M,M) 
\] is injective.  This proves $(*)$.

Now we can prove $(c).$  Let $a \in J$ be an element of degree less than or equal to $d.$   By the commutativity of \eqref{E:commute} one has 
\[
\res (m_{1}(a)) = m_{2}(\res_{\C}(a))=0.
\]   Note, however, that the map $m_{1}$ is grading preserving so $m_{1}(a)$ is of degree no more than $d.$  By $(*),$ $\res$ is injective in this range so $m_{1}(a)=0.$  That is, $a \in  I_{m|n}(M).$   Since $J$ is generated by such elements, it follows that $J \subseteq I_{m|n}(M).$ 
\end{proof}

\subsection{Support Varieties for Simple Supermodules}\label{SS:maintheorem}  In this 
section the support varieties for the simple supermodules of $\gl (m|n)$ will be 
computed.   We continue with our fixed choice of a block of atypicality $k$, 
$\mathcal{F}(m|n)_{\chi_{0}}$, and sequences $\mu_{1}, \mu_{2}, \dotsc;$ $\chi_{1}, \chi_{2}, 
\dotsc ;$ and $T_{1}, T_{2}, \dotsc.$   

Before proceeding we first make an observation which will allow us to reduce to the situation of 
Proposition~\ref{P:diagrams}.  The main concern is to ensure that we are working in $\mathcal{F}_{\chi_{i}}^{\mu_{i}}$ rather than $\mathcal{F}_{\chi_{i}}$ so that Proposition~\ref{P:diagrams}(c) can be applied.  As mentioned after \eqref{E:Phidef2}, if $M$ is an object of 
$\mathcal{F}(m|n)_{\chi_{0}}$ then for some $ N > 0$ (depending on $M$) one has that 
\begin{equation}\label{E:applyingTs}
(T_{i} \circ \dotsb \circ T_{1})M \in \mathcal{F}(m|n)_{\chi_{i}}^{\mu_{i}}
\end{equation} for all $i \geq N.$

Since the translation functors $T_{i}$ are exact, one can lift \eqref{E:applyingTs} to $\Ext$ groups as follows.  Fix $t \geq 0.$  Say 
\[
\mathcal{E}:= \left(0 \to M \to M_{1} \to \dotsb \to M_{t} \to  M \to 0 \right)
\] represents an element of $\Ext^{t}_{\mathcal{F}(m|n)_{\chi_{0}}}(M ,M).$  Then, by \eqref{E:applyingTs}, one can choose $N >0$ (depending on $M_{1}, \dotsc , M_{t}$) so that 
\[
(T_{i} \circ \dotsb \circ T_{1})M_{r} \in \mathcal{F}(m|n)_{\chi_{i}}^{\mu_{i}}
\] for all all $i \geq N$ and $r = 1, \dotsc ,t.$  That is, the induced linear map 
\[
t_{i} \circ \dotsb \circ t_{1}: \Ext^{t}_{\mathcal{F}(m|n)_{\chi_{0}}}(M ,M) \to 
\Ext^{t}_{\mathcal{F}(m|n)_{\chi_{i}}}((T_{i} \circ \dotsb \circ T_{1})M , (T_{i} \circ \dotsb \circ T_{1})M)
\] satisfies 
\begin{equation}\label{E:TsonanExt}
t_{i} \circ \dotsb  \circ t_{1} (\mathcal{E}) \in \Ext^{t}_{\mathcal{F}(m|n)_{\chi_{i}}^{\mu_{i}}}
((T_{i} \circ \dotsb \circ T_{1})M, (T_{i} \circ \dotsb \circ T_{1})M).
\end{equation}
Since $ \Ext^{t}_{\mathcal{F}(m|n)_{\chi_{0}}}(M ,M)$ is finite dimensional we can fix a basis $\mathcal{E}_{1}, \dotsc , \mathcal{E}_{q}.$  Applying \eqref{E:TsonanExt}, one can choose $N >0$ (depending on $\mathcal{E}_{1}, \dotsc , \mathcal{E}_{q}$) so that the induced map 
\[
t_{i} \circ \dotsb \circ t_{1}: 
\Ext^{t}_{\mathcal{F}(m|n)_{\chi_{0}}}(M ,M) \to 
\Ext^{t}_{\mathcal{F}(m|n)_{\chi_{i}}}((T_{i} \circ \dotsb \circ T_{1})M, (T_{i} \circ \dotsb \circ T_{1})M)
\] satisfies 
\begin{align}\label{E:VariousExts}
t_{i} \circ \dotsb \circ t_{1}\left( \Ext^{t}_{\mathcal{F}(m|n)_{\chi_{0}}}(M ,M)\right) & 
\subseteq  \Ext^{t}_{\mathcal{F}(m|n)_{\chi_{i}}^{\mu_{i}}}
((T_{i} \circ \dotsb \circ T_{1})M, (T_{i} \circ \dotsb \circ T_{1})M) \\
 & \subseteq 
\Ext^{t}_{\mathcal{F}(m|n)_{\chi_{i}}}((T_{i} \circ \dotsb \circ T_{1})M, 
(T_{i} \circ \dotsb \circ T_{1})M) \notag
\end{align}
for all $i \geq N.$  However, by assumption each of the translation functors $T_{j}$ is 
an equivalence of categories and, hence, each of the induced maps $t_{j}$ is an isomorphism.  That is, 
\[t_{i} \circ \dotsb \circ t_{1}: 
\Ext^{t}_{\mathcal{F}(m|n)_{\chi_{0}}}(M ,M) \xrightarrow{\simeq} 
\Ext^{t}_{\mathcal{F}(m|n)_{\chi_{i}}}((T_{i} \circ \dotsb \circ T_{1})M, (T_{i} \circ \dotsb \circ T_{1})M).
\]  Therefore, \eqref{E:VariousExts} becomes 
\begin{align}\label{E:VariousExts2}
t_{i} \circ \dotsb \circ t_{1}: \Ext^{t}_{\mathcal{F}(m|n)_{\chi_{0}}}(M ,M) & \xrightarrow{\simeq}  
\Ext^{t}_{\mathcal{F}(m|n)_{\chi_{i}}^{\mu_{i}}}((T_{i} \circ \dotsb \circ T_{1})M, 
(T_{i} \circ \dotsb \circ T_{1})M) \\
 & = \Ext^{t}_{\mathcal{F}(m|n)_{\chi_{i}}}((T_{i} \circ \dotsb \circ T_{1})M, 
(T_{i} \circ \dotsb \circ T_{1})M) \notag
\end{align} for all $i \geq N.$

We can finally state the result we require.  Given a fixed $d \geq 0,$ one can use 
\eqref{E:VariousExts2} to choose $N >0$ (depending on $M$ and $d$) so that \begin{align}\label{E:VariousExts3}
t_{i} \circ \dotsb \circ t_{1}: 
\Ext^{t}_{\mathcal{F}(m|n)_{\chi_{0}}}(M ,M) & \xrightarrow{\simeq}  
\Ext^{t}_{\mathcal{F}(m|n)_{\chi_{i}}^{\mu_{i}}}((T_{i} \circ \dotsb \circ T_{1})M, 
(T_{i} \circ \dotsb \circ T_{1})M) \\
 & = \Ext^{t}_{\mathcal{F}(m|n)_{\chi_{i}}}((T_{i} \circ \dotsb \circ T_{1})M, 
(T_{i} \circ \dotsb \circ T_{1})M) \notag
\end{align} for $t=0, \dotsc ,d$ and all $i \geq N.$  We now 
prove one of the main results of the paper.

\begin{theorem}\label{T:typeA}  Let $\fg =\gl (m|n),$ $r=\operatorname{def}(\fg),$ and $L(\lambda)$ be a simple $\fg$-supermodule of atypicality $k.$   Let $\widetilde{\fe}_{\1} \subseteq \fe_{\1}$ be chosen as in Section~\ref{SS:resindfunctors}. Then, 

\begin{enumerate}
\item [(a)] 
\begin{equation}\label{E:gvariety}
\operatorname{res}^{*}(\widetilde{\fe}_{\1})=\operatorname{res}^{*}\left(\mathcal{V}_{(\fe, \fe_{\0})}\left( L(\lambda)\right) \right) = \mathcal{V}_{(\fg, \fg_{\0})}\left(L(\lambda) \right) \cong \mathbb{A}^{k}. 
\end{equation}
\item [(b)]  
\begin{equation}\label{E:evariety1}
\mathcal{V}_{(\fe ,\fe_{\0})}(L(\lambda)) = W \cdot \widetilde{\fe}_{\1}.
\end{equation}
In particular, $\mathcal{V}_{(\fe ,\fe_{\0})}(L(\lambda))$ is the union of $\binom{r}{k}$ $k$-dimensional subspaces.  In terms of coordinates, if $x_{1}, \dotsc , x_{r}$ is the distinguished basis of $\fe_{\1}$ given in \eqref{E:ebasis}, then 
\begin{equation}\label{E:evariety2}
\mathcal{V}_{(\fe ,\fe_{\0})}(L(\lambda)) = \left\{ \sum_{t=1}^{r}a_{t}x_{t} \;\vert\; a_{t} \in \C \text{ and at least $r-k$ of the $a_{t}$ are zero} \right\}.
\end{equation}

\end{enumerate} 

\end{theorem}

\begin{proof} One proves (a) as follows.  First note that by the results of 
Section~\ref{SS:translationfunctors} one may assume without a loss of generality 
that $L(\lambda)$ lies in the block of atypicality $k,$ $\mathcal{F}(m|n)_{\chi_{0}},$ 
fixed before Proposition~\ref{P:diagrams}.  By using Proposition~\ref{P:translationfunctors} 
and~\eqref{E:applyingTs} we can replace $L(\lambda)$ with $(T_{i} \circ \dotsb \circ T_{1})L(\lambda)$ 
for sufficently large $i$ and 
assume without loss that $L(\lambda)$ is a simple supermodule of atypicality $k$ which lies in 
$\mathcal{F}(m|n)^{\mu_{i}}_{\chi_{i}}.$   Furthermore, by \eqref{E:VariousExts3} we can further 
assume (choosing an even larger $i$ if necessary) that  $$\Ext^{t}_{\mathcal{F}(m|n)_{\chi}}(L(\lambda),
L(\lambda)) = \Ext^{t}_{\mathcal{F}(m|n)_{\chi}^{\mu_{i}}}(L(\lambda),L(\lambda))$$  for $t=0, \dotsc, d;$ 
here, as in Proposition~\ref{P:diagrams}, $d$ is fixed so that $\operatorname{Ker}(\res_{\C})$ is generated by elements of degree 
no more than $d.$

Now by Proposition~\ref{P:diagrams}(c) we have that 
$\operatorname{Ker}(\res_{\C}) \subseteq I_{m|n}(L(\lambda)).$   On the other hand, it follows by the commutativity of \eqref{E:commute} and the injectivity of $m_{2}$ (Proposition~\ref{P:diagrams}(b)) that $I_{m|n}(L(\lambda)) \subseteq \operatorname{Ker}(\res_{\C}).$  Therefore, $I_{m|n}(L(\lambda)) = \operatorname{Ker}(\res_{\C}).$  Using the surjectivity of $\res_{\C}$ and the description of $\HH^{\bullet}(\gl(k|k), \gl (k|k)_{\0 }; \C))$ given in \eqref{E:cohomringiso}, one has 
\begin{align*}
\mathcal{V}_{(\gl (m|n), \gl (m|n)_{\0})}(L(\lambda)) &\cong \operatorname{MaxSpec}\left(\HH^{\bullet}(\gl(m|n), \gl (m|n)_{\0 }; \C)/\operatorname{Ker}(\res_{\C}) \right) \\
&\cong \operatorname{MaxSpec}\left(\HH^{\bullet}(\gl(k|k), \gl (k|k)_{\0 }; \C) \right) \\
& \cong \operatorname{MaxSpec}\left(\C[\dot{X}_{1}, \dotsc , \dot{X}_{k}] \right) \\
&\cong \mathbb{A}^{k}.
\end{align*}

Now consider $\mathcal{V}_{(\fe, \fe_{\0})}(L(\lambda)).$  Recall that $\tilde{\fe}_{\1} \subseteq \fe_{\1}.$  Since $L(\lambda) \in \mathcal{F}(m|n)^{\mu_{i}}_{\chi_{i}},$ it follows from 
Proposition~\ref{P:serganovaskeylemma} that $L(\lambda)$ contains a 
simple $\gl(k|k)$-supermodule of atypicality $k$ as a direct summand (namely, 
$\Res_{\mu_{i}}L(\lambda)$).  By Proposition~\ref{L:fullatypicality} and the rank variety 
description of $\widetilde{\fe}$ support varieties discussed in Section~\ref{SS:detecting}, 
it must be that for any $x \in \tilde{\fe}_{\1},$ $L(\lambda)$ is not projective as an $\langle 
x \rangle$-supermodule. Here $\langle x \rangle$ denotes the Lie subsuperalgebra generated by $x.$  
This statement is equally true if we view $x$ as an element of $\fe_{\1}.$  Thus, we have $\widetilde{\fe}_{\1} \subseteq \mathcal{V}_{(\fe , \fe_{\0})}(L(\lambda)).$ Therefore by \eqref{E:resmapsbetweenvarieties} one has,  
\begin{equation}\label{E:anotherdamnequation}
\res^{*}(\widetilde{\fe}_{\1}) \subseteq \res^{*}(\mathcal{V}_{(\fe, \fe_{\0})}(L(\lambda)) \subseteq \mathcal{V}_{(\fg, \fg_{\0})}(L(\lambda)) \cong \mathbb{A}^{k}.
\end{equation}
However, by \eqref{E:resmapsbetweenvarieties} the map $\res^{*}$ is finite-to-one so $\res^{*}\left( \widetilde{\fe}_{\1}\right)$ is a $k$-dimensional closed subset of $\mathbb{A}^{k}.$  However $\mathbb{A}^{k}$ is a $k$-dimensional irreducible variety.  Therefore $\res^{*}\left( \widetilde{\fe}_{\1}\right)=\mathbb{A}^{k}$ and all the containments in \eqref{E:anotherdamnequation} must be equalities.  This proves $(a).$

To prove $(b),$ one recalls from \cite[(6.1.3)]{BKN} that the fibers of the map $\res^{*}$ are precisely the orbits of $W.$  This along with \eqref{E:anotherdamnequation} implies \eqref{E:evariety1}.  To obtain \eqref{E:evariety2}, one uses \eqref{E:evariety1} and the explicit description of the action of $W$ on $X_{1}, \dotsc, X_{r}$ provided in Section~\ref{SS:eforgl}.
\end{proof}

\subsection{}\label{SS:conclusions} Note that the above theorem confirms several 
conjectures for the simple supermodules of $\gl (m|n)$.  The first observation is that the 
second equality in \eqref{E:gvariety} affirms a speculation of the authors in \cite[Section 6.2]{BKN}.  

Second, observe that 
\[
\dim \mathcal{V}_{(\fe ,\fe_{\0})}(L(\lambda)) = k = \operatorname{atyp}(L(\lambda)).
\] This agrees with \cite[Conjecture 7.2.1]{BKN}, where it was conjectured that the dimension of the 
$\fe$ support variety of a simple supermodule should equal its atypicality.  
The verification of this conjecture justifies the general 
(and functorial) definition of atypicality for finite dimensional supermodules of $\gl (m|n)$
by setting $\operatorname{atyp}(M):=\dim \mathcal{V}_{(\fe ,\fe_{\0})}(M)$ for 
all $M$ in ${\mathcal F}(m|n)$. 

Finally recall that the superdimension of a supermodule $M$ is the integer 
$\dim M_{\0}-\dim M_{\1}.$  In \cite[Conjecture 3.1]{kacwakimoto} Kac and Wakimoto conjectured that 
for a simple basic classical Lie superalgebra, $\fg$, the superdimension of a simple supermodule $L$ 
is nonzero if and only if $\operatorname{atyp}(L) = \operatorname{def}(\fg).$  As was discussed in 
\cite[Section 7.3]{BKN}, the validity of \cite[Conjecture 7.2.1]{BKN} proved above for $\gl(m|n)$ implies 
the ``only if'' direction of the Kac--Wakimoto conjecture in this case.  It should be noted 
that this direction of the conjecture was also recently proved in \cite{dufloserganova}.

\section{Clifford Superalgebras, Superdimension, and Divisibility}\label{S:clifford}

\subsection{}\label{SS:subalg}  In this section we show that the codimension of the $\fe$ 
support variety of a supermodule $M$ is closely related to the $2$-divisibility of the dimension of $M$.  
Another consequence is that if $\mathcal{V}_{(\fe ,\fe_{\0})}(M)$ is a proper subset of 
$\mathcal{V}_{(\fe , \fe_{\0})}(\C),$ then the superdimension of $M$ is necessarily zero.  
First we establish some general results.

\noindent 
Let $\fc=\fc_{\0} \oplus \fc_{\1}$ be a Lie superalgebra which satisfies the following assumptions:  
\begin{enumerate}
\item  $\fc_{\0}=[\fc_{\1},\fc_{\1}]$ and is an abelian Lie algebra;
\item $[\fc_{\0}, \fc_{\1}]=0$. 
\end{enumerate}

Let $\mathcal{F}=\mathcal{F}(\fc, \fc_{\0})$ be the category of all finite dimensional $\fc$-supermodules which are finitely semisimple as $\fc_{\0}$-supermodules. Note that if $\fc_{\1}$ is a subspace of $\fe_{\1}$ 
(where $\fe$ is the detecting subalgebra discussed in Section~\ref{SS:detecting}), then $\fc := [\fc_{\1},\fc_{\1}] \oplus \fc_{\1}$ satisfies the above assumptions.  This is the context where the general theory 
will be applied.

Given  $\chi \in \fc_{\0}^{*},$ let $\C_{\chi}$ denote the unique simple $\fc_{\0}$-supermodule of 
dimension $1$ (concentrated in degree $\0$) with action $x.v=\chi(x)v$ for all $x \in \fc_{\0}$ and 
$v \in \C_{\chi}.$  Given  $\chi \in \fc_{\0}^{*},$ let $\mathcal{F}_{\chi}$ denote the full subcategory of $\mathcal{F}$ consisting of all supermodules $M$ such that all composition factors are isomorphic to $\C_{\chi}$ when viewed as a $\fc_{\0}$-supermodule by restriction.  By \cite[Lemma~5.1.2]{BKN} one has the following decomposition of the category $\mathcal{F},$ 
\[
\mathcal{F} = \bigoplus_{\chi \in \fc_{\0}^{*}} \mathcal{F}_{\chi}.
\]

\subsection{Clifford Superalgebras} Given a $\chi \in \fc_{\0}^{*}$ one can define a Clifford superalgebra as follows.  Define a bilinear form  
\[
(\;,\;): \fc_{\1} \otimes \fc_{\1} \to \C
\]
by 
\[
(x,y) = \chi([x,y]).
\]  Since $[x,y]=[y,x]$ for all $x,y \in \fc_{\1},$ this bilinear form is symmetric.  Let $A_{\chi}$ denote the Clifford superalgebra defined by 
\begin{equation}\label{E:cliffordalgdef}
A_{\chi}= T(\fc_{\1})/I_{\chi},
\end{equation} where $T(\fc_{\1})$ denotes the tensor superalgebra on the superspace $\fc_{\1}$ and $I_{\chi}$ is the ideal of $T(\fc_{\1})$ generated by the set
\[
\setof{ x\otimes y + y \otimes x - (x,y)}{x,y \in \fc_{\1}}.
\]  The $\Z_{2}$-grading on $A_{\chi}$ is obtained by setting $\p{x}=\1$ for all $x \in \fc_{\1}.$  Write $A_{\chi}$-smod for the category of all finite dimensional $A_{\chi}$-supermodules.  The key result is the following proposition which was used by Penkov \cite{penkov} in his study of the Lie superalgebra $\mathfrak{q}(n)$ (see also \cite[Section 2.2]{frisk}).

\begin{prop}\label{P:categoryequiv}  The category $\mathcal{F}_{\chi}$ is isomorphic to the category $A_{\chi}$-smod.
\end{prop}  

\begin{proof}  Let $M$ be an object in $\mathcal{F}_{\chi}.$  Since $\fc_{\1}$ generates $A_{\chi},$ and for $x,y \in \fc_{\1}$ and $m \in M$ one has 
\[
x.(y.m)+y.(x.m) = [x,y].m = \chi ([x,y])m = (x,y)m,
\] one has a well defined action of $A_{\chi}$ on $M.$  The $\Z_{2}$-grading on $M$ is compatible with this action so in fact $M$ is an $A_{\chi}$-supermodule.

Conversely, let $M$ be an $A_{\chi}$-supermodule.  We then have an action by $\fc_{\1}$ on $M$ via the canonical inclusion $\fc_{\1} \hookrightarrow A_{\chi}.$  This extends to an action of $\fc_{\0}$ by having $[x,y] \in \fc_{\0}$ act by $[x,y].m=\chi([x,y])m$ for all $m \in M.$  It is straightforward to verify that this action along with its given $\Z_{2}$-grading makes $M$ into a $\fc$-supermodule.

One also can verify that a linear map which defines a supermodule homomorphism in one category defines a homomorphism in the other category.  That is, the morphisms in the two categories concide.
\end{proof}

Recall that Schur's Lemma in this setting states that $\dim\Hom_{\fc}(L,L)$ equals $1$ or $2$ for 
any simple $\fc$-supermodule $L$ in $\mathcal{F}.$  We say $L$ is type \texttt{Q} if the 
$\operatorname{Hom}$-space is two dimensional, and type \texttt{M} otherwise.  

\begin{prop}\label{P:factsaboutFchi} Let $\chi \in \fc_{\0}^{*}.$  Set $z=\dim\ \{\, x \in \fc_{\1} \;\vert\; (x,y)=0 \text{ for all } y \in \fc_{\1}\,\}$, $n=\dim \fc_{\1} - z$,  and $\widetilde{n}=\lfloor (n+1)/2 \rfloor.$ Then the following statements hold for the category $\mathcal{F}_{\chi}.$
\begin{enumerate}
\item [(a)] There is a unique simple supermodule in $\mathcal{F}_{\chi},$ which we denote by $S(\chi),$ of type $\texttt{M}$ if $n$ is even, and of type $\texttt{Q}$ otherwise.  The supermodule $S(\chi)$ has dimension $2^{\widetilde{n}}.$  If $n >0,$ then $S(\chi)$ has superdimension $0$.
\item [(b)] Write $P(\chi)$ for the projective cover of $S(\chi).$ Then $P(\chi)$ is also injective and of dimension 
\[
\dim P(\chi)= \begin{cases}  2^{\dim(\fc_{\1})-\widetilde{n}}, & \text{ if $S(\chi)$ is of type $\texttt{M};$}\\
 2^{\dim(\fc_{\1})-\widetilde{n}+1}, & \text{ if $S(\chi)$ is of type $\texttt{Q}.$}
\end{cases}
\]  Furthermore, $P(\chi)$ has superdimension zero.
\end{enumerate}

\end{prop}

\begin{proof}  By applying Proposition~\ref{P:categoryequiv}, the results in part (a) and the statement on injectivity in part (b) follow from \cite{brundankleshchev, penkov,frisk}.  It only remains to calculate the dimension of $P(\chi).$  

Since $\C_{\chi}$ is projective in $\mathcal{F}(\fc_{\0},\fc_{\0})$ and the induction functor is exact, 
$U(\fc)\otimes_{U(\fc_{\0})}\C_{\chi}$ is projective in $\mathcal{F}_{\chi}.$   Hence, 
\[
\dim U(\fc)\otimes_{U(\fc_{\0})}\C_{\chi} = a \dim P(\chi) 
\] for some positive integer $a.$  By Frobenius reciprocity, one has
\begin{align*}
\dim \Hom_{\fc}(U(\fc)\otimes_{U(\fc_{\0})}\C_{\chi}, S(\chi)) =  \dim \Hom_{\fc_{\0}}(\C_{\chi}, S(\chi))=\dim S(\chi)  = 2^{\widetilde{n}}.
\end{align*}  On the other hand,
\begin{align*}
\dim \Hom_{\fc}(U(\fc)\otimes_{U(\fc_{\0})}\C_{\chi}, S(\chi)) &= a \dim \Hom_{\fc}(P(\chi), S(\chi)) \\
    &=  \begin{cases} a, & \text{ if $S(\chi)$ is of type $\texttt{M};$}\\
                                                    2a, & \text{ if $S(\chi)$ is of type $\texttt{Q}.$} \end{cases} 
\end{align*}

By the PBW theorem $\dim U(\fc)\otimes_{U(\fc_{\0})}\C_{\chi}=2^{\dim(\fc_{\1})}.$  Combining this with the above calculations one obtains 
\begin{equation}\label{E:Pchidim}
\dim P(\chi)= \begin{cases}  2^{\dim(\fc_{\1})-\widetilde{n}}, & \text{ if $S(\chi)$ is of type $\texttt{M};$}\\
 2^{\dim(\fc_{\1})-\widetilde{n}+1}, & \text{ if $S(\chi)$ is of type $\texttt{Q}.$}
\end{cases}
\end{equation}
Lastly, if $n>0,$ then the simple supermodule in $\mathcal{F}_{\chi}$ has superdimension zero and so all supermodules in $\mathcal{F}_{\chi}$ also have superdimension $0$.  In particular, this holds for $P(\chi)$.  If $n=0$ then $P(\chi)=U(\fc) \otimes_{U(\fc_{\0})}\C_{\chi},$ which has superdimension $0.$  
\end{proof}

\subsection{\bf $2$-Divisibility}\label{S:supportvarieties}  Let $\fe$ be a detecting Lie superalgebra 
for ${\mathfrak g}$ as discussed in Section~\ref{SS:detecting}. The following theorem relates the codimension of 
the support variety of a supermodule in $\mathcal{F}(\fe, \fe_{\0})$ to the $2$-divisibility of its 
dimension and to its superdimension. 

\begin{theorem}\label{T:codimtheorem}  Let ${\mathfrak g}$ be classical Lie superalgebra with detecting subalgebra $\fe.$  Let
$M$ be an object of $\mathcal{F}(\fe, \fe_{\0})$ and let 
\[
d= \dim\mathcal{V}_{(\e,\e_{\0})}(\C)-\dim\mathcal{V}_{(\e,\e_{\0})}(M)
\]
denote the codimension of the variety $\mathcal{V}_{(\e,\e_{\0})}(M).$  Then, 
\[
2^{\lfloor d/2 \rfloor} \;\vert\; \dim M.
\]  Furthermore, if $d > 0,$ then $M$ has superdimension $0.$
\end{theorem}

\begin{proof}  The case $d=0$ is vacuously true so we assume $d>0.$ As a consequence of the Noether Normalization Theorem (e.g.\ \cite[Theorem 3.1]{Kunz}) one can choose a subspace 
$H \subseteq  \mathcal{V}_{(\e,\e_{\0})}(\C)=\fe_{\1}$ of dimension $d$ such that 
$H \cap \mathcal{V}_{(\e,\e_{\0})}(M)= \{0 \}.$  Let $\fc$ denote the Lie subsuperalgebra of 
$\fe$ generated by $H;$ that is, $\fc_{\1}=H$ and $\fc_{\0}=[H,H].$  Observe that $\fc$ is a Lie superalgebra 
of the type considered in Sections 5.1-5.2.  Using the rank variety description (cf.\ Section~\ref{SS:detecting}) 
and the fact that $\fc_{\1} \cap \mathcal{V}_{(\e,\e_{\0})}(M)= \{0 \},$ one has $\mathcal{V}_{(\fc , \fc_{\0})}(M)=0.$  It then follows by \cite[Theorem 6.4.2(b)]{BKN} that $M$ is projective as a $\fc$-supermodule.  Thus $M$ decomposes 
as a direct sum of $P(\chi)$ 
for various $\chi \in \fc_{\0}^{*}.$  The dimension of each $P(\chi)$ is a power of two 
and is at its smallest when $n=d=\dim(\fc_{\1})$ and $S(\chi)$ is of type \texttt{M}.  Therefore, one has
\[
\dim P(\chi) \geq 2^{\lfloor d/2 \rfloor}.
\]  The first statement of the theorem follows from this inequality.  
Furthermore, $M$ has superdimension $0$ because each $P(\chi)$ has superdimension $0$. 
\end{proof}

\begin{corollary}\label{C:divisibilityforglmn}  Let $L(\lambda)$ be the finite dimensional simple $\gl(m|n)$-supermodule of highest weight $\lambda$, let $r=\operatorname{min}(m,n),$ the defect of $\gl(m|n),$ and let $a=\operatorname{atyp}(L(\lambda))$.  Then 
\[
2^{\lfloor (r-a)/2 \rfloor} \; \vert \; \dim L(\lambda).
\]  Furthermore, if the atypicality is strictly less than the defect, then the superdimension of $L(\lambda)$ is zero.

\end{corollary}

\begin{proof}  By \cite[Section 8.8]{BKN} and Theorem~\ref{T:typeA}, one has that $r= \dim \mathcal{V}_{(\e,\e_{\0})}(\C)$ and $a= \dim \mathcal{V}_{(\e,\e_{\0})}(L(\lambda))$.  The result then follows by the previous theorem.
\end{proof}

We remark that a closed formula for the dimension of $L(\lambda)$ is given by Su and Zhang \cite[Theorem 4.14]{suzhang}.  However, their formula is quite complicated and it does not seem that the above divisibility statement can easily be recovered from their work.


\def\Dbar{\leavevmode\lower.6ex\hbox to 0pt{\hskip-.23ex \accent"16\hss}D}
  \def\cprime{$'$}

\end{document}